\documentclass[12pt, reqno]{amsart}

\usepackage[usenames,dvipsnames]{xcolor}
\usepackage{amssymb, amsmath, amsthm}
\usepackage[backref, colorlinks=true, citecolor=Green, urlcolor=blue]{hyperref}
\usepackage[alphabetic,backrefs,lite]{amsrefs}
\usepackage{amscd}   % for commutative diagrams
\usepackage{fullpage}
\usepackage{tikz-cd}
\usepackage{colonequals}
\usepackage{MnSymbol}
\usepackage{multirow}
\usepackage{enumitem}

\DeclareFontEncoding{OT2}{}{} % to enable usage of cyrillic fonts

\newtheorem{lemma}{Lemma}[section]
\newtheorem{theorem}[lemma]{Theorem}

\newtheorem{prop}[lemma]{Proposition}
\newtheorem{cor}[lemma]{Corollary}
\newtheorem{conj}[lemma]{Conjecture}

\newtheorem{claim*}{Claim}
\newtheorem*{theorem*}{Theorem}
\newtheorem{thm}[lemma]{Theorem}

\theoremstyle{remark}
\newtheorem{remark}[lemma]{Remark}
\newtheorem{remarks}[lemma]{Remarks}

\newtheorem{example}[lemma]{Example}
\newtheorem{examples}[lemma]{Examples}
\newtheorem{defn}[lemma]{Definition}
\newtheorem{definition}[lemma]{Definition}

\newcommand{\A}{{\mathbb A}}

\newcommand{\PP}{{\mathbb P}}
\newcommand{\C}{{\mathbb C}}
\newcommand{\F}{{\mathbb F}}
\newcommand{\Q}{{\mathbb Q}}

\newcommand{\Z}{{\mathbb Z}}

\newcommand{\kbar}{{\overline{k}}}

\newcommand{\kk}{{\mathbf k}}

\newcommand{\calC}{{\mathcal C}}

\newcommand{\calE}{{\mathcal E}}

\newcommand{\calI}{{\mathcal I}}

\newcommand{\calO}{{\mathcal O}}

\newcommand{\calT}{{\mathcal T}}

\newcommand{\calV}{{\mathcal V}}

\newcommand{\OO}{{\mathcal O}}

\newcommand{\frakm}{{\mathfrak m}}

\newcommand{\frakp}{{\mathfrak p}}

\newcommand{\mm}{{\mathfrak m}}
\newcommand{\be}{{\mathbf{e}}}

\newcommand{\bff}{{\mathbf{f}}}
\newcommand{\bw}{{\mathbf{w}}}

\usepackage[mathscr]{euscript}

\newcommand{\scrS}{{\mathscr S}}

\DeclareMathOperator{\HH}{H}

\DeclareMathOperator{\lcm}{lcm}

\DeclareMathOperator{\rk}{rk}

\DeclareMathOperator{\Aut}{Aut}

\DeclareMathOperator{\ord}{ord}
\DeclareMathOperator{\Sym}{Sym}

\DeclareMathOperator{\Pic}{Pic}

\DeclareMathOperator{\Spec}{Spec}

\DeclareMathOperator{\Frac}{Frac}

\DeclareMathOperator{\tors}{tors}

\DeclareMathOperator{\rank}{rank}

\DeclareMathOperator{\Disc}{Disc}

\DeclareMathOperator{\cond}{cond}

\DeclareMathOperator{\Supp}{Supp}
\DeclareMathOperator{\spp}{\mathsf{P}}
\DeclareMathOperator{\Diag}{Diag}
\DeclareMathOperator{\indx}{index}
\DeclareMathOperator{\pr}{pr}
\DeclareMathOperator{\mat}{Mat}

\numberwithin{equation}{section}
\numberwithin{table}{section}

\newcommand{\defi}[1]{\textsf{#1}} % for defined terms

\title{Likely and unlikely ramification of algebraic points on curves}

\author{Bianca Viray}
\address{University of Washington, Department of Mathematics, Box 354350, Seattle, WA 98195, USA}
\email{bviray@uw.edu}
\urladdr{http://math.washington.edu/\~{}bviray}

\author{Isabel Vogt}
\address{Brown University, Department of Mathematics, Box 1917, 151 Thayer Street, Providence, RI 02912, USA}
\email{ivogt.math@gmail.com}
\urladdr{https://www.math.brown.edu/ivogt/}
\thanks{BV was partially supported by NSF DMS-2401327. IV was partially supported by NSF DMS-2338345.}

\begin{document}
%%%%%%%%%%%%%%%%%%%%%%%%%%%%%%%%%%%%%%%%%%%%%%%%%%%%%%%%%%%%%%%%%%%%%%%%%%%%%%%%%%%%%%%%%%%%%%%%%%%%%%%%%%
%%%%%%%%%%%%%%%%%%%%%%%%%%%%%%%%%%%%%%%%%%%%%%%%%%%%%%%%%%%%%%%%%%%%%%%%%%%%%%%%%%%%%%%%%%%%%%%%%%%%%%%%%%
\begin{abstract}
    This paper initiates the study of arithmetic properties of closed points  of arbitrary degree on curves. We show that a pointless curve of  rank \(0\) can have only finitely many cyclic degree \(\ell\) points for \emph{all} sufficiently large primes \(\ell\). More generally, we show that the residue fields of points on curves of genus at least \(2\) and rank \(0\) are often biased away from total ramification. The proof technique illuminates a natural intermediate conjecture between the Manin--Mumford conjecture and the Zilber--Pink conjecture for self-products of abelian varieties. This intermediate conjecture implies analogous results for other splitting profiles.
\end{abstract}
\maketitle

\section{Introduction}

    A central focus in arithmetic geometry is understanding how the geometry of a variety influences its arithmetic. For example, the Mordell conjecture (proved by Faltings) asserts that curves of genus at least \(2\) can have only finitely many points defined over any fixed number field \(k\). In other words, for curves of genus at least \(2\), the map
    \begin{equation}\label{eq:PointsToResidueFields}
    \{x \in C : x \textup{ closed}\} \to \{L/k : [L:k] <\infty\}, \quad x \mapsto \kk(x)
    \end{equation}
    sending a point \(x\) to its residue field \(\kk(x)\) is finite-to-1. (Equivalently, this map sends a Galois orbit of geometric points to the isomorphism class of a minimal field of definition of a point in the orbit.) 
    One can also ask about the image of the map, and the philosophy that geometry controls arithmetic suggests that for a curve of genus at least \(2\) the image of~\eqref{eq:PointsToResidueFields} should be constrained.
    
    We explore this question and discover that, for many curves, large ramification degrees are unlikely to appear in residue fields of closed points. For example, we prove the following.

\begin{theorem}\label{thm:IntroTotallyRamified}
    Let \(C/k\) be a nice curve whose Jacobian has rank \(0\). Assume that there 
    is no geometric point \(x \in C(\bar{k})\) with \(\indx(C/k) x\) linearly equivalent to a \(k\)-rational divisor.
    Then there exists an integer \({m}\) such that for every \(d\) coprime to \(m\), total ramification is \textbf{unlikely} for degree \(d\cdot \textup{index}(C/k)\) points on \(C\).
    Precisely, the set of places
    \[
        \{v\in \Omega_k : \exists x\in C, \deg(x) = d\cdot\textup{index}(C/k), \kk(x)/k \textup{ totally ramified at } v\}
    \]
    is finite.
\end{theorem}

The hypotheses of Theorem~\ref{thm:IntroTotallyRamified} are sharp; see Remarks~\ref{rmks:Thm1}\ref{rmksthm1_4}. If we further restrict to points on \(C\) of \emph{prime} degree, Theorem~\ref{thm:IntroTotallyRamified} has the following consequence for cyclic points.

\begin{cor}\label{cor:FinitelyManyCyclicIntro}
    Let \(C/k\) be a nice curve whose Jacobian has rank \(0\). If \(C(k) = \emptyset\), then for all sufficiently large primes \(\ell\), there are finitely many cyclic degree \(\ell\) points on \(C\).
\end{cor}
\begin{proof}
    The conclusion holds vacuously if \(\indx(C/k) >1 \), since then there are no points of large prime degree.  We may therefore assume that \(\indx(C/k) = 1\).
    It suffices to consider only those primes \(\ell\) that are unramified in \(k\).
    The assumptions of the Corollary mean that we may apply Theorem~\ref{thm:IntroTotallyRamified} to conclude that for all primes \(\ell\) coprime to \(m\) (in particular, all sufficiently large primes), total ramification for degree \(\ell\) points can only occur at a finite set of places \(\calT_{\ell}\). In a cyclic extension of prime degree, any ramified prime must be totally ramified. Thus, if \(\kk(x)\) is cyclic of degree \(\ell>2\) it must be contained in the ray class field of modulus \(\prod_{v\in \calT_\ell} v^{n_v}\), where \(n_v = 1\) if \(v\nmid \ell\) and \(n_v = 2\) if \(v|\ell\) (using the hypothesis that \(\ell\) is unramified in \(k\)). This ray class field is a number field, so by Faltings's (Curve) Theorem there are finitely many points of \(C\) defined over this ray class field. 
\end{proof}

\begin{remarks}\label{rmks:Thm1}\hfill
    \begin{enumerate}[label=(\roman*)]
        \item By definition of the index of the curve, if \(d\not \equiv 0 \pmod{\indx(C)}\), then \(C\) has \emph{no} closed points of degree \(d\). For example, if \(C(\mathbb{R})=\emptyset\), then the index of \(C\) must be even and \(C\) has no closed points of odd degree.
        \item If \(\textup{genus}(C)\leq 1\) , then every degree \(d\) divisor is linearly equivalent to \(dp\), for a geometric point \(p \in C(\bar{k})\).  Hence, the assumptions of Theorem~\ref{thm:IntroTotallyRamified} force \(\textup{genus}(C)\geq 2\).
        \item For \(\PP^1\), the map~\eqref{eq:PointsToResidueFields} is surjective. Since any curve \(C\) has some map \(f\colon C \to \PP^1\), this implies that for any number field \(L\) and any curve \(C\), there is some field \(F\) in the image of~\eqref{eq:PointsToResidueFields} that contains \(L\).
        \item\label{rmksthm1_4} 
        If there exists a geometric point \(x \in C_{\bar{k}}\) such that \(\indx(C/k)x\) is linearly equivalent to a \(k\)-rational divisor, then for all \(d \geq (2g(C)+1)/\indx(C)\),
        the set of places in Theorem~\ref{thm:IntroTotallyRamified} has positive density (see Corollary~\ref{cor:Thm1.1Sharp}).
        In particular, the hypotheses of Theorem~\ref{thm:IntroTotallyRamified} for curves whose Jacobian has rank \(0\) are sharp in the strongest possible sense. 
        \item 
For many curves satisfying the assumptions of Theorem~\ref{thm:IntroTotallyRamified}, e.g., curve \href{https://www.lmfdb.org/Genus2Curve/Q/7676/a/7676/1}{7676.a.7676.1}, 
we can show that the theorem holds with \(m = 2\), so odd total ramification is \emph{always} unlikely. (In particular, for these curves, ``sufficiently large primes \(\ell\)'' in Corollary~\ref{cor:FinitelyManyCyclicIntro} means \(\ell\) odd.) For small odd degrees (\(e = 3, 5\)), we can compute the finite set of places of good reduction where degree \(e\) points can be totally ramified; see Table~\ref{table:genus2}. 
    \end{enumerate}
\end{remarks}

We obtain results analogous to Theorem~\ref{thm:IntroTotallyRamified} for other ramification indices.
In order to precisely state our results in this direction, we introduce some notation.

\begin{definition}\hfill
\begin{enumerate}[label=(\roman*)]
\item A \defi{splitting profile} \(\spp\) is a finite multiset of ordered pairs
\(\spp = \{(e_1, f_1), \dots, (e_r, f_r)\}\).  Define \(r(\spp)\colonequals |\spp|\) and \(\deg(\spp) \colonequals \sum_{i=1}^r e_if_i\).  From a splitting profile \(\spp\), we can make two length \(r(\spp)\) ordered tuples, well-defined up to simultaneous permutation:
\[\be(\spp) \colonequals (e_1, \dots, e_r), \qquad \bff(\spp) \colonequals (f_1, \dots, f_r).\]
\item A {splitting profile} \(\spp\) is \defi{unramified} if \(\be(\spp) = (1, \dots, 1)\), and \defi{ramified} otherwise.
\item Given a finite extension \(L/k\) and a nonarchimedean place \(v\) of \(k\), the ramification and residual degrees of the places lying over \(v\) give a splitting profile of degree \([L:k]\):
\[
    \spp_v(L/k) \colonequals\left\{ (e(L_w/k_v), [\F_w:\F_v]) : w|v\right\}.
\]
\item A splitting profile \defi{\(\spp\) is likely in \(V\)}, a set of effective divisors on \(C\), if 
\[
    \# \left\{v : \exists x\in V \textup{ integral, with }\spp_v(\kk(x)/k) = \spp
    \right\} = \infty,
\]
and \defi{unlikely} otherwise.
We say that \defi{\(\spp\) is likely/unlikely on \(C\)} when \(V\) is taken to be a set that contains all closed points of degree equal to \(\deg(\spp)\).
\end{enumerate}
\end{definition}

\begin{remarks}\label{rmks:Likely}\hfill
\begin{enumerate}[label=(\roman*)]
    \item If \(\spp\) is an unramified splitting profile and there exists a point \(x\in C\) and a place \(v\) such that \(\spp_v(\kk(x)/k) = \spp\), then \(\spp\) is likely in any set containing \(x\) by the Chebotarev density theorem.
    \item If the set \(\{\kk(x) : x \in V \textup{ integral}\}\) is finite and \(\spp\) is a ramified splitting profile, then \(\spp\) is unlikely in \(V\). 
    \item Since \(k\) is dense in \(\prod_{v\in S} k_v\) for any finite set of places \(S\), there exist monic irreducible polynomials over \(k\) that have prescribed \(v\)-adic factorization at any finite set of places \(v\). Therefore, given a splitting profile \(\spp\) and a finite set of places \(S\), one can construct a number field \(L/k\) with \(\spp_v(L/k) = \spp\) for all \(v\in S\).  For example, an extension totally ramified at all places in \(S\) can be obtained by adjoining a root of a polynomial \(f(x)\) that is Eisenstein at \(v\) for all \(v\in S\).
    Since every number field occurs as the residue field of a closed point on \(\PP^1\), every splitting profile is likely on \(\PP^1\). \label{part:likelyP1}
    \item Likely splitting profiles in complete linear systems tend to propagate under taking multiples and repetitions. For example, if an unramified splitting profile \(\spp^0\) is likely in a linear system \(|D|\), then sufficiently large multiples \(d\spp^0 = \{{(d,f) : f\in \bff(\spp^0)}\}\) are likely in \(|dD|\) (Corollary~\ref{cor:MultipleOfRealizedUnramIsRealized}); see Section~\ref{sec:LikelyExamples} for more details. 
    \end{enumerate}
\end{remarks}

 \begin{theorem}\label{thm:IntroRank1Splitting}
    Let \(C/k\) be a nice curve whose Jacobian has rank \(0\), and let \(\spp^0 = \{(e_i, f_i)\}_i\) be a splitting profile such that \(\indx(C)|e_if_i\) for all \(i\) and such that \(\frac{\deg(\spp^0)}{\indx(C)}\) is relatively prime to \(\#J_C(k)\).
    If \(C\) has no points of degree at most \(\deg(\spp^0)\), then there exists an integer \(m = m(C, \spp^0)\) such that for every integer \(d\) coprime to \(m\), the splitting profile \(\spp = d\spp^0 =  \{ (de_i, f_i)\}_i\) is \textbf{unlikely} on \(C\).
\end{theorem}

\begin{remark}
    Theorem~\ref{thm:IntroTotallyRamified} can be rephrased as saying if \(\spp^0 := \{(\indx(C), 1)\}\) is unlikely, then arbitrarily large multiples \(d \spp^0 \colonequals \{(de,f) : (e,f)\in \spp^0\}\) are also unlikely. A similar rephrasing of Theorem~\ref{thm:IntroRank1Splitting} is possible, but one instead must assume that \(\spp^0\) and all of its refinements are unlikely. See Section~\ref{sec:LikelyPositiveDensity} for more details.
\end{remark}

A key input to the proofs of Theorems~\ref{thm:IntroTotallyRamified} and \ref{thm:IntroRank1Splitting} is the Manin--Mumford conjecture (proved by Raynaud). We obtain similar statements for splitting profiles where the ramification indices are allowed to vary over a group of rank at least \(2\), conditional on a generalization of the Manin--Mumford conjecture (also a specialization of the Zilber--Pink conjecture). 

\begin{theorem}\label{thm:IntroUnlikelySplitting}
    Let \(C/k\) be a nice curve whose Jacobian has rank \(0\) and is geometrically simple. Let \(\spp^0 = \{(e_i^0, f_i)\}\) be a splitting profile. Assume:
    \begin{itemize}
        \item the Zilber--Pink conjecture (or the weaker Conjecture~\ref{conj:GeneralizedMM}),
        \item \(\sum_i f_i < g(C)\),
        \item \(\indx(C) \mid e_i^0f_i\) for all \(i\),
        \item \(\min_i(e_i^0f_i)/\indx(C)\) is relatively prime to \(\#J_C(k)\), and
        \item there are no points of degree at most \(\min_i(e_i^0f_i)\) on \(C\).
    \end{itemize}
    Then, for all \(\be \in \prod_i e_i^0\Z_{>0}\) outside of a proper finite union of subgroups,  the splitting profile \(\mathsf{P} = \{(e_i,f_i)\}_i\) is \textbf{unlikely} on \(C\).
\end{theorem}

\begin{remark}
    Since \(\prod_i e_i\Z\) is finitely generated, the complement of a proper finite union of subgroups must include a coset of a finite index subgroup~\cite{VirayVogt}*{Lemma 4.3.8}.  Assume the Zilber--Pink conjecture. If \(J_C\) is geometrically simple, \(J_C(k) = \{O\}\), and \(\indx(C) = 1\), then for each \(\bff = (f_1, \dots, f_r)\) for which \(\sum_i f_i < g\) and \(C\) has no points of degree at most \(\min_i(f_i)\), Theorem~\ref{thm:IntroUnlikelySplitting}  implies that a positive density set of splitting profiles \(\spp\) with \(\bff(\spp) = \bff\) are unlikely.
\end{remark}

As is typically expected in the unlikely intersections framework, there are some geometric conditions that can cause a splitting profile to be likely. The assumptions that \(\sum_i f_i < g(C)\), and that there are no points of degree at most \(\min_i(e_i^0f_i)\) exclude these scenarios. If the Jacobian of \(C\) has nontrivial isogeny factors then the assumptions required are more delicate. For the most general version of Theorem~\ref{thm:IntroUnlikelySplitting}, see Theorem~\ref{thm:unlikely_splitting}.

The methods of this paper apply more generally to determine whether particular splitting profiles are unlikely in the linear systems \(\{|dD|\}_{d\in \Z_{>0}}\) for \(D\) a fixed divisor of positive degree (See Section~\ref{sec:UnlikelySplittingProfiles} for more details on this). For a curve \(C/k\) that has index \(1\) and whose Jacobian has rank \(0\), the union of linear systems 
\[
\bigcup_{D\in \Pic^1_C(k)}\bigcup_{d\in \Z_{>0}} |dD|
\]
contains all closed points on \(C\) whose degree is coprime to \(\#J_C(k)\), and thus we can deduce statements like Theorems~\ref{thm:IntroTotallyRamified}
,~\ref{thm:IntroRank1Splitting}, 
and~\ref{thm:IntroUnlikelySplitting}.

The method of proof also gives an effective (but not necessarily efficient) method to compute the finite set of places where an unlikely splitting profile can be realized. We compute these finite sets for 20 genus \(2\) curves in the case of total ramification in degrees \(3\) and \(5\); they are listed below in Table~\ref{table:genus2}. The proof that Table~\ref{table:genus2} is correct is in Section~\ref{sec:ProofOfTable} and the associated code is in the \href{https://github.com/b-viray/unlikely-ramification/}{Github repository} for this paper. (We compute these finite sets using a short Weierstrass model for \(C\), so this model has bad reduction at \(2\).)

%%%%%%%%%%%%%%%%%%%%%%%%%%%%%%%%%%%%%%%%%%%%%%%%%%%%%%%%%%%%%%%%%%%%%%%%%%%%%%%%
%%%%%%%%%%%%%%%%%%				Tables			%%%%%%%%%%%%%%%%%%%%%%%%
%%%%%%%%%%%%%%%%%%%%%%%%%%%%%%%%%%%%%%%%%%%%%%%%%%%%%%%%%%%%%%%%%%%%%%%%%%%%%%%%

\begin{table}[h]
\begin{tabular}{|r|r|r|}
\hline 
LMFDB label      &  \multicolumn{2}{|c|}{{\(\left\{p>2 : {\exists x\in C,\; \spp_p(\kk(x)/k) = \spp, \textup{ and }C\textup{ has good reduction at }p}\right\}^{\phantom{1}}\)}}                  
\\[10pt]
 &  \multicolumn{1}{|c|}{ \(\spp = \{(3,1)\}\)} & \multicolumn{1}{|c|}{\(\spp = \{(5,1)\}\)}\\
\hline
\href{https://www.lmfdb.org/Genus2Curve/Q/7676/a/7676/1}{7676.a.7676.1}    
            &                       & 316004230293431                       \\
\hline
\href{https://www.lmfdb.org/Genus2Curve/Q/9977/a/9977/1}{9977.a.9977.1}
            &  2027                 &     2837                               \\
\hline
\href{https://www.lmfdb.org/Genus2Curve/Q/11259/a/11259/1}{11259.a.11259.1}
            &                       & 232695223753                      \\
\hline
\href{https://www.lmfdb.org/Genus2Curve/Q/13072/a/13072/1}{13072.a.13072.1}
            &                  & 10301, 23625124907                         \\
\hline
\href{https://www.lmfdb.org/Genus2Curve/Q/15504/a/62016/1}{15504.a.62016.1}
            &                       & 11, 6451166627922673              \\
\hline
\href{https://www.lmfdb.org/Genus2Curve/Q/19170/c/76680/1}{19170.c.76680.1}
            & 35153            &                                         \\
\hline
\href{https://www.lmfdb.org/Genus2Curve/Q/21989/a/21989/1}{21989.a.21989.1}
            &  3                    &   1685914369                         \\
\hline
\href{https://www.lmfdb.org/Genus2Curve/Q/22032/a/22032/1}{22032.a.22032.1}
            &                       &   7054572663424903                      \\
\hline
\href{https://www.lmfdb.org/Genus2Curve/Q/22185/a/22185/1}{22185.a.22185.1}
            &                       &  4903622393                        \\
\hline
\href{https://www.lmfdb.org/Genus2Curve/Q/23033/a/23033/1}{23033.a.23033.1}
            &                       &  1277, 19181, 492731             \\
\hline
\href{https://www.lmfdb.org/Genus2Curve/Q/23735/a/23735/1}{23735.a.23735.1}
            &                       &  225889                       \\
\hline
\href{https://www.lmfdb.org/Genus2Curve/Q/25707/a/77121/1}{25707.a.77121.1}  
            &                       & 199, 28882395713187018090626537        \\
\hline
\href{https://www.lmfdb.org/Genus2Curve/Q/29636/a/474176/1}{29636.a.474176.1}
            &                       & 6719                        \\
\hline
\href{https://www.lmfdb.org/Genus2Curve/Q/31174/a/342914/1}{31174.a.342914.1}
            &                       &  1044737, 553763689903, 10546179748889         \\
\hline
\href{https://www.lmfdb.org/Genus2Curve/Q/31668/a/63336/1}{31668.a.63336.1}
            &                       &  11                                     \\
\hline
\href{https://www.lmfdb.org/Genus2Curve/Q/36114/a/36114/1}{36114.a.36114.1}
            &                       & 5, 1091                              \\
\hline
\href{https://www.lmfdb.org/Genus2Curve/Q/41151/a/41151/1}{41151.a.41151.1}
            &                       & 41, 131, 293, 12647, 175891                  \\
\hline
\href{https://www.lmfdb.org/Genus2Curve/Q/43335/b/43335/1}{43335.b.43335.1}
            &                       & 17, 6583714939              \\
\hline
\href{https://www.lmfdb.org/Genus2Curve/Q/45175/a/45175/1}{45175.a.45175.1}  
            &   3                   & 88339, 5807603, 95190432083453372821    \\
\hline
\href{https://www.lmfdb.org/Genus2Curve/Q/47449/a/47449/1}{47449.a.47449.1}  
            &                       & 431, 2939, 880057   \\
\hline
\end{tabular}
\caption{Genus 2 curves that satisfy Theorem~\ref{thm:IntroTotallyRamified} with \(m=2\), so odd total ramification is unlikely. For \(e = 3, 5\), we list the finite set of odd primes of good reduction that are  totally ramified in the residue field of some degree \(e\) point on \(C\). See Section~\ref{sec:ProofOfTable} for more details.}
\label{table:genus2}
\end{table}

%%%%%%%%%%%%%%%%%%%%%%%%%%%%%%%%%%%%%%%%%%%%%%%%%%%%

The results mentioned thus far focus on criteria for when a given splitting profile is unlikely. We also explore the complementary question of when splitting profiles are likely. For example, we show analogously to the case when \(C=\PP^1\) mentioned in Remark~\ref{rmks:Likely}\ref{part:likelyP1}, all splitting profiles are likely on an elliptic curve. 

\begin{prop}\label{prop:IntroOnEAllProfilesLikely}
    Let \(E\) be an elliptic curve over a number field \(k\). For any \(x\in E(k)\) and any splitting profile \(\spp\), the set
    \[
    \{v\in \Omega_k :  \exists y\in E, y\sim \deg(\spp)x, \spp_v(\kk(y)/k) = \spp\}
    \]
    contains all sufficiently large primes. In particular, every splitting profile with \(\deg(\spp) =d\) is likely in the complete linear system \(|{d}x|\).
\end{prop}
The key ingredient in the proof is a lifting result, which may be of independent interest (Theorem~\ref{thm:Reconstruction}): given an effective divisor over \(\F_v\) in a sufficiently positive linear system, we construct a closed point \(x\in C_k\) whose splitting profile is determined by the components of the effective divisor over \(\F_v\). This lifting result also allows us to show in some situations that a likely splitting profile is realized by a positive density set of places. For example, in the context of Theorem~\ref{thm:IntroTotallyRamified} we prove the following.

\begin{theorem}\label{thm:IntroLikelyFrobenian}
    Let \(C/k\) be a nice curve whose Jacobian has rank \(0\). For any \(d \geq 2g(C) + 1\), the set of places
    \[
        \left\{v\in \Omega_k : \exists x\in C, \deg(x) = d, \kk(x)/k \text{ totally ramified at \(v\)}\right\}   
    \]
    is Frobenian. In particular, if the set is infinite, it has positive density.
\end{theorem}
See Corollary~\ref{cor:MaximalLikelyIsPositiveDensity} for the most general version of this statement.

%%%%%%%%%%%%%%%%%%%%%%%%%%%%%%%%%%%%%%%%%%%%%%%%%%%%
%%%%%%%%%%%%%%%%%%%%%%%%%%%%%%%%%%%%%%%%%%%%%%%%%%%%

%%%%%%%%%%%%%%%%%%%%%%%%%%%%%%%%%%%%%%%%%%%%%%%%%%%%
%%%%%%%%%%%%%%%%%%%%%%%%%%%%%%%%%%%%%%%%%%%%%%%%%%%%

\subsection{Related work}
\subsubsection{\(\ell\)-Diophantine Stability} In~\cite{MazurRubin}, Mazur and Rubin introduce the notion of \(\ell\)-Diophantine stability: given a prime integer \(\ell\), a variety \(X\) over a global field \(k\) is \defi{\(\ell\)-Diophantine-stable} if for every \(n\) and every finite set of places \(\Sigma \subset \Omega_k\), there are infinitely many cyclic extensions \(L/k\) of degree \(\ell^n\) completely split at all places in \(\Sigma\) and such that \(X\) has no points whose residue fields are contained in \(L\) and strictly larger than \(k\) (in other words, \(X(L) = X(k)\)). Among other results, they prove that every curve of genus at least \(1\) whose endomorphism ring is equal to its geometric endomorphism ring is \(\ell\)-Diophantine-stable for a positive density set of \(\ell\)~\cite{MazurRubin}*{Theorem 1.3}. This result is similar, but not directly comparable, to Corollary~\ref{cor:FinitelyManyCyclicIntro}, even for a fixed curve that satisfies both hypotheses. Indeed, 
Corollary~\ref{cor:FinitelyManyCyclicIntro} gives a stronger restriction on cyclic points of \emph{prime} degree than~\cite{MazurRubin}*{Theorem 1.3}. However, our results do not constrain cyclic points with degree a nontrivial power of a prime, since cyclic extensions of degree \(\ell^r\) with \(r>1\) can have ramified primes that are not totally ramified. On the other hand,~\cite{MazurRubin}*{Theorem 1.3} rules out infinitely many degree \(\ell^r\) cyclic residue fields for a positive proportion of primes \(\ell\).

\begin{remark}
    If one is interested in studying closed points on a curve \(C\) whose residue field is Galois with Galois group \(G\), then the relevant moduli space is \(C^{|G|}/G\). If \(G\) is cyclic of order at least \(3\) and \(C\) has genus at least \(2\), then \(C^{|G|}/G\) is general type~\cite{Rawson-Solvable}*{Theorem 1}, so the Bombieri--Lang conjecture predicts that the rational points are not Zariski dense. 
\end{remark}

\subsubsection{Distributions of number fields} If a nice curve \(C\) satisfies the assumptions of Theorem~\ref{thm:IntroTotallyRamified}, then, since there exist number fields totally ramified at any prime (see Remark~\ref{rmks:Likely}\ref{part:likelyP1}), there are infinitely many number fields of degree \(n \colonequals d\indx(C)\) that \emph{cannot} be the residue field of a point on \(C\) whenever \(d\) is coprime to \(m_C\). In fact, the Bhargava--Malle heuristics suggest that this excluded infinite set of number fields is a positive proportion of all \(S_{{n}}\) number fields when \(n \geq 3\), as we now explain.

In~\cite{Bhargava-LocalMass}*{Page 10}, Bhargava outlines the following heuristic for the distribution of number fields: \textit{the expected (weighted) number of global \(S_{n}\)-number fields of discriminant \(D\) is simply the (infinite) product of the (weighted) number of local extensions (i.e., \'etale algebras \(K/\Q_p\) of degree \(n\)) that are discriminant-compatible with \(D\).} This heuristic predicts (see~\cite{Bhargava-LocalMass}*{(4.2)}) that for a finite set of primes \(S\) of \(\Q\), the proportion of \(S_{n}\) number fields that have ramification indices strictly less than \({n}\) for all primes lying over a prime \(p\notin S\) is:
\begin{equation}\label{eq:proportion}
    \prod_{p\notin S} \left(1 - \frac{\sum_{[K:\Q_p]={n}, \textup{totally ramified}} \frac{1}{\Disc_p(K)}\cdot \frac1{\#\Aut(K)}}{\sum_{[K:\Q_p]={n}} \frac{1}{\Disc_p(K)}\cdot \frac1{\#\Aut(K)}}\right).
\end{equation}

Serre showed that~\cite{Serre-MassFormula}*{Th\'eor\`eme 2}
\[
\sum_{\substack{[K\colon\mathbb{Q}_p] =n\\ K/\Q_p \textup{ totally ramified}}} \frac{1}{\Disc_p(K)}\cdot\frac{1}{\#\Aut(K)} = \frac{1}{p^{n-1}}.
\]
and Bhargava showed that~\cite{Bhargava-LocalMass}*{Theorem 1.1}
\[
\sum_{[K:\Q_p]={n}} \frac{1}{\Disc_p(K)}\cdot \frac1{\#\Aut(K)} = \sum_{k=0}^{n-1} q(k, n-k)p^{-k},
\]
where \(q(k, n-k)\) denotes the number of partitions of \(k\) into at most \(n-k\) parts. Thus, ~\eqref{eq:proportion} becomes
\[
 \prod_{p\notin S} \left(1 - \frac{1}{\sum_{k=0}^{n-1}q(k, n-k)p^{n-1-k}}\right).
\]
Since \(q(0, n) = 1\), the above product is bounded above by \(\prod_{p\notin S'}\left(1 - \frac{1}{cp^{n-1}}\right)\) for a (possibly larger) finite set \(S'\) and a positive constant \(c\). For \(n\geq 3\), this upper bound converges to a constant in the open interval \((0,1)\). Therefore, the heuristic predicts that the set of fields excluded by Theorem~\ref{thm:IntroTotallyRamified} is a positive proportion of all \(S_{n}\) fields.

\subsubsection{Unlikely intersections} As mentioned above, a key input in the proof of Theorem~\ref{thm:IntroRank1Splitting} is the Manin--Mumford conjecture (proved by Raynaud), which describes the structure of \(Z\cap A_{\tors}\), where \(Z\subset A\) is a closed subvariety of an abelian variety. For the more general, albeit conditional, Theorem~\ref{thm:IntroUnlikelySplitting}, the relevant intersection is
\[
    Z\cap \bigcup_{\mathbf{e}} A^r[\mathbf{e}],
\]
where the union ranges over some subset of positive integer vectors \(\mathbf{e}\in \Z_{>0}^r\), \(Z\subset A^r\) is a closed subvariety of a self-product of an abelian variety, and \(A^r[\mathbf{e}]\) denotes the kernel of the morphism \(A^r \xrightarrow{[\mathbf{e}]} A, \;(p_1, \dots,p_r)\mapsto \sum_i e_ip_i\). Specializing the Zilber--Pink conjecture to intersections of this form gives Conjecture~\ref{conj:GeneralizedMM}, which we contend is a natural intermediate generalization of the Manin--Mumford conjecture. (The case when the union ranges over a rank \(1\) subgroup of \(\Z^r\) can be reduced to the Manin--Mumford conjecture.) To the best of our knowledge, Conjecture~\ref{conj:GeneralizedMM} is open, even in the simplest case where we obtain interesting applications: when \(\dim A = 3\), \(r=2=\dim Z\) and \(A\) is geometrically simple.

\subsubsection{Splitting profiles in basepoint free pencils}
    The problem of determining which splitting profiles appear as \(\spp_v(\kk(x)/k)\) for \(x\) ranging over a basepoint free pencil, i.e., fibers of a morphism \(C\to \PP^1\), has been explored extensively for Galois morphisms. Indeed, works of Beckmann, D\`ebes, Ghazi, K\"onig, Legrand, and Neftin~\cites{Beckmann, DG2012,  DL-TAMS, KLN-GrunwaldRamified, Legrand} give a complete answer to this question (for all but finitely many places), and explore its applications to the Grunwald and parameterization problems. In the non-Galois case, recent work of Bal\c{c}ik, Chan, Liu, and the first author~\cite{BCLV} studies the splitting profiles \(\spp_v(\kk(C_t)/k)\) and the isomorphism classes of \(\kk(C_t)\otimes_k k_v\) for all but finitely many places. The results in these papers give necessary and sufficient conditions for a splitting profile to be realized for all places where the \emph{morphism} \(C\to \PP^1\) has good reduction. This may exclude finitely many primes in addition to the primes where the curve has bad reduction.

\subsubsection{Algebraic integral points on punctured curves}
As discussed in Remarks~\ref{rmks:Thm1}, curves of genus at most 1 never satisfy the hypotheses of Theorem~\ref{thm:IntroTotallyRamified}, and, indeed, on a curve of genus at most 1, the set of places of total ramification given in Theorem~\ref{thm:IntroLikelyFrobenian} has positive density for all sufficiently large \(d\) divisible by the index. This sharp contrast in behavior for curves of genus at least \(2\) and curves of genus at most \(1\) echoes the contrast seen in Faltings's (curve) theorem: a curve has infinitely many \(k\)-points for some number field \(k\) \emph{if and only if} its genus is at most \(1\). However, one can obtain a statement with no assumption on genus by combining Faltings theorem with Siegel's Theorem and allowing for punctured curves: Any hyperbolic curve has only finitely many \(S\) integral points, for any number field \(k\) and finite set of places \(S\). Thus, one may ask if an analogue of Theorem~\ref{thm:IntroTotallyRamified} holds for punctured curves.  This line of research was pursued by Siksek for punctured genus \(1\) curves~\cite{Siksek-PuncturedAVs} and Freitas--Kraus--Siksek for \(S\)-unit equations~\cite{FKS-Unit}. The proofs in this paper and those of Siksek, Freitas--Kraus--Siksek are suggestive of each other, although, to the best of our knowledge, there are no direct implications in either direction.

\subsection{Outline}
    We begin in Section~\ref{sec:MMConsequences} setting up the unlikely intersections framework needed to prove our main theorems. As described above, the intersections relevant to our problem are of the form
    \[
       Z\cap \bigcup_{\mathbf{e}\in\mathbb{Z}^r} A^r[\mathbf{e}],
    \]
    where \(Z\subset A^r\) is a closed subvariety of a self-product of an abelian variety, and \(A^r[\mathbf{e}]\) denotes the kernel of the morphism \(A^r \xrightarrow{[\mathbf{e}]} A, \;(p_1, \dots,p_r)\mapsto \sum_i e_ip_i\). While the Zilber--Pink conjecture puts forth a necessary criterion for this intersection to be Zariski dense in \(Z\), we need to extract from this a condition on \(\mathbf{e}\) that guarantees that the intersection \(Z\cap A^r[\mathbf{e}]\) is \emph{empty}. We show that for \(\mathbf{e}\) away from a Zariski closed subset, the intersections \(Z\cap A^r[\mathbf{e}]\) are contained in the union of finitely many torsion translates, each of which has larger than expected intersection with \(Z\) (relative to the subgroups \(A^r[\mathbf{e}]\)). The characterizations of torsion translates and Zariski closed subsets become simpler when \(A\) has distinct isogeny factors; see Theorem~\ref{thm:ConsequenceZP_general} for the precise statement. These results may be of independent interest.
    
    In Section~\ref{sec:UnlikelySplittingProfiles}, we lay out the general strategy to determine when splitting profiles are unlikely in a complete linear system. The most general versions of our main theorems about unlikely splitting profiles are Theorems~\ref{thm:unlikely_splitting_rank1} and~\ref{thm:unlikely_splitting}. Theorems~\ref{thm:IntroTotallyRamified},~\ref{thm:IntroRank1Splitting} and~\ref{thm:IntroUnlikelySplitting} are obtained as corollaries of Theorems~\ref{thm:unlikely_splitting_rank1} and~\ref{thm:unlikely_splitting} as explained in Sections~\ref{sec:CorOfUnlikely_rank1} and~\ref{sec:CorOfUnlikelySplitting}. In Section~\ref{sec:ProofOfTable}, we explain how we computed the information in Table~\ref{table:genus2} and raise questions about finer properties of the sets of places realizing total ramification.

    In Section~\ref{sec:Reconstruction}, we explore converses to the results in Section~\ref{sec:UnlikelySplittingProfiles} for very ample linear systems.  The main theorem of this section is Theorem~\ref{thm:Reconstruction}, which morally can be viewed as an extension of Hensel's Lemma and Hilbert irreducibility to the case of non-smooth subschemes. From this theorem, we give several simple conditions that each guarantee likely ramification (Section~\ref{sec:LikelyExamples}) which together give a proof of Proposition~\ref{prop:IntroOnEAllProfilesLikely} (Section~\ref{sec:ProofOfIntroOnE}).

    We close in Section~\ref{sec:LikelyPositiveDensity} with a study of the set of places realizing a given splitting profile, including the proof of Theorem~\ref{thm:IntroLikelyFrobenian}. 

    \begin{remark}
        For readability, we have endeavored to make the sections as independent as possible. For example, Sections~\ref{sec:MMConsequences} and~\ref{sec:Reconstruction} (with the exception of Lemma~\ref{lem:ScalarMultOfRealizedIsRealized}) can each be read independently of all other sections. A reader willing to accept Proposition~\ref{prop:consequenceMM} and Theorem~\ref{thm:ConsequenceZPDistinctIsogeny} may read Section~\ref{sec:UnlikelySplittingProfiles} independently of all other sections. Finally, while all of Section~\ref{sec:overview} provides helpful context for Section~\ref{sec:LikelyPositiveDensity}, the arguments depend only on Lemma~\ref{lem:SplProfileImpliesFvPoints} and  Theorem~\ref{thm:Reconstruction}.
    \end{remark}

    \subsection{Notation}
    We use \(k\) to denote a number field. For any finite set of places \(S\), we write \(\OO_{k, S}\subset k\) for the ring of \(S\)-integers. Given a nonarchimedean place \(v\), we write \(\F_v\) for the residue field of \(v\).
    
    A nice curve \(C\) is a smooth projective geometrically integral scheme of dimension \(1\) over a field. Given a nice curve \(C\), we write \(g(C)\) for its genus, \(\Pic_C\) for its Picard scheme, and \(J_C\)  for the identity component (its Jacobian). Given a divisor \(D\) on \(C\) we write \([D]\) to denote its divisor class (modulo linear equivalence) and, given another divisor \(E\) on \(C\), we write \(E \sim D\) if \([E] = [D]\).  We write \(|D|\) for the complete linear system of \(D\). More generally, for any subvariety \(W\subset \Pic^d_C\), we write \(|W|\) to denote the set of effective divisors whose linear equivalence class is contained in \(W\).

    For a positive integer \(f\), we write \(\Sym^f_C\) for the \(f^{th}\) symmetric product and \(W^f_C\subset \Pic^f_C\) for the locus of effective degree \(f\) divisors. Given a vector of positive integers \(\mathbf{f}\in \Z_{>0}^r\), we write \(W^{\mathbf{f}}\) for the product \(W^{f_1}\times \cdots \times W^{f_r}\). For any divisor \(D\) of positive degree, translation by a suitable multiple of \(D\) gives an isomorphism between a component of the Picard scheme containing divisors with degree divisible by \(\deg(D)\) and the identity component \(J_C\). 

    More generally, if \(\bff\) is a tuple and \(D\) is a divisor of positive degree with \(\deg(D)|f_i\) for all \(i\), then we write \(\iota_D\) for the morphism
    \[
        \Pic^{f_1}_C\times \cdots \times\Pic^{f_r}_C \xrightarrow{\sim} J_C^{r}, \quad (z_1, \dots, z_r)\mapsto \left(z_1 - \textstyle\frac{f_1}{\deg(D)}D, \dots, z_r - \textstyle\frac{f_r}{\deg(D)}D\right).
    \]

    Given an abelian group \(A\) and a tuple of positive integers \(\mathbf{w}\in \Z_{>0}^r\), componentwise multiplication (alternatively multiplication by a diagonal matrix with integer entries) gives a morphism that we denote
    \[
        [\Diag(\bw)]\colon A^r\to A^r, \quad (z_1, \dots, z_r) \mapsto (w_1z_1, \dots w_rz_r).
    \]
    If \(A\) is torsionfree (e.g., \(A = \Z\)), then \([\Diag(\bw)]\) is injective. Under this assumption, given an element \(\bw'\in [\Diag(\bw)]A^r\), we write \(\bw'/\bw\) for the unique preimage 
    \([\Diag(\bw)]^{-1}(\bw')\).
    We write \(\sigma\colon A^r\to A\) for the summation map \(\sigma(z_1, \dots, z_r) = z_1 + \dots + z_r\), and we write \([\bw]\) for the composition \(\sigma\circ [\Diag(\bw)]\), which can be thought of as a dot product map. We denote the kernel of this morphism by \(A^r[\bw]\).

    Given a smooth proper integral model \(\calC \to \Spec \OO_{k,S}\) with generic fiber \(C_{k} = C\), the relative Picard variety \(\Pic_{\calC/\OO_{k,S}}\) is then a group scheme over \(\Spec \OO_{k,S}\) and every component \(\Pic^d_{\calC/\OO_{k,S}}\) is smooth and proper over \(\Spec \OO_{k,S}\). Given a closed subscheme \(H \subset \Pic_{C/k}\), we write \(H_{\F_v}\) for the fiber of the closure of \(H\) in \(\Pic_{\calC/\OO_{k,S}}\) over \(v\). (This is nonempty by the valuative criterion for properness.) In particular, there exists a unique limit of any \(k\)-divisor class \(D\) modulo \(v\) for every \(v\in \Spec \OO_{k,S}\), which we denote \(D_{\F_v}\). A closed point \(x \in C\) determines a map \(\Spec \OO_{\kk(x){, S}} \to \calC\).

   %%%%%%%%%%%%%%%%%%%%%%%%%%%%%%%%%%%%%%%%%%%%%%%%%%%%%%%%%%%%%%%%%%%%%%%%%%%%%%%%%%%%%%%%%%%%%%%%%%%%%%%%%%
\section*{Acknowledgements}
    We thank Laura Capuano, Jordan Ellenberg, Aaron Levin, Carlo Pagano, Joe Silverman, Jakob Stix, and Melanie Matchett Wood for  conversations that helped contextualize these results.  We thank James Rawson for pointing out a mistake in an earlier draft of this paper. We acknowledge the use of Claude to obtain a preliminary referee report.

%%%%%%%%%%%%%%%%%%%%%%%%%%%%%%%%%%%%%%%%%%%%%%%%%%%%%%%%%%%%%%%%%%%%%%%%%%%%%%%%%%%%%%%%%%%%%%%%%%%%%%%%%%

%%%%%%%%%%%%%%%%%%%%%%%%%%%%%%%%%%%%%%%%%%%%%%%%%%%%%%%%%%%%%%%%%%%%%%%%%%%%%%%%%%%%%%%%%%%%%%%%%%%%%%%%%%
%%%%%%%%%%%%%%%%%%%%%%%%%%%%%%%%%%%%%%%%%%%%%%%%%%%%%%%%%%%%%%%%%%%%%%%%%%%%%%%%%%%%%%%%%%%%%%%%%%%%%%%%%%
\section{A generalization of the Manin--Mumford Conjecture}\label{sec:MMConsequences}
%%%%%%%%%%%%%%%%%%%%%%%%%%%%%%%%%%%%%%%%%%%%%%%%%%%%%%%%%%%%%%%%%%%%%%%%%%%%%%%%%%%%%%%%%%%%%%%%%%%%%%%%%%

The questions and results considered here are geometric, so in this section we work over~\(\C\).
The Manin--Mumford Conjecture (now a theorem of Raynaud) states that the Zariski closure of the torsion points on a subvariety \(Z\subset A\) of an abelian variety \(A\) decomposes as a finite union of torsion translates contained in \(Z\)~\citelist{\cite{Raynaud-MMcurves}*{Th\'eor\`eme I}, \cite{Raynaud-MMgeneral}*{Th\'eor\`eme, page 327}}. As a consequence, if \(Z\) does not contain \(0\), then \(Z\cap A[m] = \emptyset\) for a positive density set of positive integers.

\begin{prop}\label{prop:consequenceMM}
    Let \(A\) be an abelian variety and let \(Z \subset A\) be a closed subvariety such that \(0 \not\in Z\).  Then there exist positive integers \(m_1, \dots, m_s \in \Z_{\geq 2}\) such that
    \[
    Z \cap A[m] \neq \emptyset \;\Longleftrightarrow\; m \in \bigcup_i m_i \Z.
    \]
\end{prop}
\noindent(See Section~\ref{sec:ProofOfConsequenceMM} for the proof.)

We are interested in generalizations where we consider a subvariety of a power of an abelian variety, \(Z\subset A^r\) and its intersection with the subgroups \(A^r[\be]\) defined above. (This case arises in our study of unlikely splitting profiles in linear systems of curves; see Section~\ref{sec:overview} for details about this connection.)

Informally, the Zilber--Pink conjecture predicts that \(Z\) contains a Zariski dense set of points that are in \(\bigcup_{\mathbf{e}}A^r[\mathbf{e}]\) only if codimension estimates predict that it is likely. More precisely, we have the following.

\begin{defn}\label{defn:calEZ}
If \(Z \subset A^r\) is an integral subvariety, let \(\langle Z \rangle \subset A^r\) be the smallest algebraic subgroup of \(A^{r}\) containing \(Z\) and define
\begin{equation}\label{eq:EZ}
    \calE_Z \colonequals \{\be \in \Z^r  : \dim \langle Z \rangle \cap A^r[\be] \geq \dim \langle Z \rangle - \dim Z\}.
\end{equation}
Since \(A^r[\be]\) is the kernel of the morphism \([\be] \colon A^r \to A\), we have \(\dim \langle Z\rangle\cap A^r[\be] = \dim \langle Z\rangle - \dim [\be]\langle Z\rangle.\) Thus
    \begin{align}\label{eq:calEZ_alternate}
    \calE_Z := &\{\be : \dim \langle Z\rangle\cap A^r[\be] \geq \dim \langle Z\rangle - \dim Z\} \notag \\
    = & \{\be : \dim \langle Z\rangle - \dim [\be]\langle Z\rangle \geq \dim \langle Z\rangle - \dim Z\}\notag \\
    = & \{\be : \dim Z\geq \dim [\be]\langle Z\rangle\}.
    \end{align}
\end{defn}

\begin{conj}[Follows from the Zilber--Pink Conjecture; see Section~\ref{sec:RelationZP} for details]\label{conj:GeneralizedMM}
    Let \(A\) be an abelian variety and let \(Z \subset A^r\) be an integral subvariety.    Then
    \[Z \cap \left(\bigcup_{\be \in \Z^r \smallsetminus \calE_Z} A^r[\be]  \right) \]
    is not Zariski dense in \(Z\).
\end{conj}
\begin{remark}
    Observe that when \(r=1\), the subgroup \(A^r[\be]\)  is simply the \(e_1\)-torsion \(A[e_1]\) and so \(\dim\langle Z \rangle \cap A[e_1] \leq 0\). Additionally, \(\dim Z  = \dim \langle Z \rangle\) if and only if \(Z\) is a torsion translate. Hence, when \(r=1\), \(\calE_Z \neq \{0\}\) if and only if \(Z\) is a torsion translate and therefore the above conjecture reduces to the Manin--Mumford conjecture.
\end{remark}
Assuming this conjecture, we can prove a generalization of Proposition~\ref{prop:consequenceMM} to \(r > 1\).

\begin{theorem}\label{thm:ConsequenceZPDistinctIsogeny}
    Let \(A\) be an abelian variety and let \(Z \subset A^r\) be an integral subvariety. Assume Conjecture~\ref{conj:GeneralizedMM} and that \(A\) has distinct isogeny factors. Suppose that there exists \(\tilde{\be} \in \Z^r\) such that for every abelian subvariety \(B \subset A\) and every torsion point \(\tau\in A^r\), either \(\dim Z \cap (\tau +
    B^r) < \dim B\) or \([\tilde{\be}]\tau \not\in B\).  Then there exist finitely many proper subgroups \(M_1, \dots, M_s \subsetneq \Z^r\)  such that
    \[Z \cap \left( \bigcup_{\be \in \Z^r \smallsetminus \cup_i M_i} A^r[\be]\right) = \emptyset.\]
    The subgroups may be chosen so that additionally \(\bigcup_{i=1}^s M_i\subsetneq \Z^r\) and \(\tilde{\be}\notin M_i\) for each finite index subgroup \(M_i\).
\end{theorem}

\begin{remarks}\label{rem:thm_ZP}\hfill
\begin{enumerate}[label=(\roman*)]
    \item As mentioned in the introduction, once 
    \(\bigcup_i M_i\subsetneq\Z^r\), the complement \(\Z^r \smallsetminus \bigcup_i M_i\) must include a coset of a finite index subgroup~\cite{VirayVogt}*{Lemma 4.3.8}.
    \item\label{rem2} An \(\tilde{\be}\) satisfying the assumptions of Theorem~\ref{thm:ConsequenceZPDistinctIsogeny} will never exist if \(0 \in Z\), (the condition fails for \(\tau = \mathbf{0}, B=\{0\}\)). The conclusions also fail in this case since \(0 \in A^r[\be]\) for all \(\be\in \Z^r\).  More generally, if \(\dim Z \cap B^r \geq \dim B\) for any abelian subvariety \(B\) (including if \(\dim Z\cap A^r = \dim Z\geq \dim A\)), then such an \(\tilde{\be}\) can not exist.  See Examples~\ref{ex:genus2} below, which illustrate why this condition is necessary.
    \item Without the assumption on the existence of \(\tilde{\be}\) (and preserving the assumptions that Conjecture~\ref{conj:GeneralizedMM} holds and that \(A\) has distinct isogeny factors), we prove that for \(\be\) outside a finite union of subgroups of rank at most \(r-1\), the intersection \(Z \cap \bigcup A^r[\be]\) is constrained to lie in finitely many torsion translates (see Theorem~\ref{thm:ConsequenceZP_general}\eqref{part:ConsZP_general_1} for the precise statement).  This result combined with Lemma~\ref{lem:TorsionTranslate} below then implies that as long as \(\dim Z \cap B^r < \dim B\) for all abelian subvarieties \(B \subseteq A\), there are proper subgroups \(M_1, \dots, M_s\) such that \(Z \cap \bigcup_{\be \not\in \cup_i M_i} A^r[\be] = \emptyset\).  However, even though each \(M_i\neq\Z^r\), the union \(\cup M_i\) may cover \(\Z^r\). 
    This can indeed happen.  Let \(r=2\), \(A\) be a simple abelian surface and let \(\tau \in A[2] \smallsetminus 0\) be a nontrivial \(2\)-torsion point.  Let \(Z \subset A^2\) be an integral curve containing the torsion points \(T_1 \colonequals (\tau, 0)\), \(T_2 \colonequals (0, \tau)\), and \(T_3 \colonequals (\tau, \tau)\), and not containing \((0,0)\).
    Then \(T_1 \in A^2[\be]\) if and only if \(e_1\) is even, \(T_2 \in A^2[\be]\) if and only if \(e_2\) is even, and \(T_3 \in A^2[\be]\) if and only if \(e_1\) and \(e_2\) have the same parity.  Hence \(Z \cap A^2[\be]\) is nonempty for every \(\be \in \Z^2\).
    \item If we remove both the assumption on \(\tilde{\be}\) and on distinct isogeny factors (while still assuming Conjecture~\ref{conj:GeneralizedMM}), then we prove that for all \(\mathbf{e}\) outside a proper \emph{Zariski closed subset} \(V\subset \mathbb{C}^r\), the intersection \(Z \cap \bigcup A^r[\be]\) is again constrained to lie in a finite union of torsion translates. However, if \(A\) has repeated isogeny factors, then there is greater variation in the possible torsion translates (if there are no repeated isogeny factors then only torsion translates of product subgroups arise), and this leads to more subtle variation in the vectors \(\mathbf{e}\) which contribute to nontrivial intersections.
\end{enumerate}
\end{remarks}

\begin{examples}\label{ex:genus2}\hfill
\begin{enumerate}[label=(\roman*)]
    \item Let \(C\) be a nice genus \(2\) curve and let \(D\in \Pic^1_C \smallsetminus {W^1_C}\) be an ineffective divisor of degree \(1\).  Let \(Z_C \colonequals \iota_{D}\left(W^{(1,1)}_C \right) \subset J_C^2\).  Since \(D\notin W^1_C\), we have \(0 \notin Z_C\). On the other hand, for any \(\be\in \Z^2\) with \(e_1e_2\neq0\), the restriction of \([\be]\) to \(Z_C\) gives a surjective map \(Z_C\to J_C\). This implies that \(Z_C\cap J_C^2[\be]\neq\emptyset\). (Indeed, since the torsion in \(J_C\) is dense, we additionally have that
    \(
        Z_C\cap \bigcup_{m} J_C^2[m\be]    
    \)
    is Zariski dense in \(Z_C\).) 
    In particular, the conclusion of Theorem~\ref{thm:ConsequenceZPDistinctIsogeny} fails. This does not contradict Theorem~\ref{thm:ConsequenceZPDistinctIsogeny} because \(\dim Z_C = 2 = \dim J_C\), and so no \(\tilde{\be}\) can exist (Remark~\ref{rem:thm_ZP}\ref{rem2}).
\item \label{ex:cover_genus2}
    Retain the setup from the previous example.
    In addition, let \(X\) be a smooth curve with a morphism \(\phi\) of degree \(f\) to the genus \(2\) curve \(C\). Assume that \(X\) has genus at least \(2f+1\).  
    Let \(A \colonequals J_X\) and let \(B \colonequals \phi^*J_C\). Assume that \(D\) is such that \(\phi^*D\in \Pic^f_X\) is also ineffective.  
    Let \(Z \colonequals \iota_{{\phi^*D}}(W^{(f,f)}_X) \subset A^2\). 
    Our assumption that \(\phi^*D \not\in W^f_X\) guarantees that
    \(0 \notin Z\).  Since \(\dim Z = 2f\), our assumption that \(X\) has genus at least \(2f+1\) guarantees that \(\dim Z < \dim A\).
    
    We have a commutative diagram
    \begin{center}
        \begin{tikzcd}
         \iota_{D}\left(W^{(1,1)}_C\right) = Z_C \arrow[r, hook] \arrow[d, "\phi^*"] & J_C^2 \arrow[d, "\phi^*"] & \\
             \iota_{{\phi^*D}}\left(W^{(f,f)}_X\right) = Z \arrow[r, hook] & J_X^2 = A^2. \\
        \end{tikzcd}
    \end{center}
    Hence,
    \( \phi^*\left(Z_C \cap J_C^2[\be]\right) \subset \phi^*(Z_C) \cap B^2[\be]  \subset Z \cap A^2[\be]\).
    In particular, for any \(\be\) for which \(e_1e_2 \neq 0\), we have
    \(Z \cap A^2[\be] \neq \emptyset\).  This is consistent with Theorem~\ref{thm:ConsequenceZPDistinctIsogeny} since \(\dim (Z \cap B^2) \geq \dim \phi^*(Z_C) = 2\) is not less than \(\dim B = 2\), and so no such \(\tilde{\be}\) can exist (Remark~\ref{rem:thm_ZP}\ref{rem2}).\hfill\(\righthalfcup\)
    \end{enumerate}
\end{examples}

\subsection*{Outline of Section~\ref{sec:MMConsequences}} 
If \(Z\) is a torsion translate of a product subgroup, then we may prove Theorem~\ref{thm:ConsequenceZPDistinctIsogeny} unconditionally (and regardless of any assumption on the isogeny factors of \(A\)); we do this in Section~\ref{sec:ZPTorsion} (Lemma~\ref{lem:TorsionTranslate}).  
In Section~\ref{sec:ProofOfConsequenceMM}, we prove Proposition~\ref{prop:consequenceMM}. In Section~\ref{sec:EZ}, we prove structural results about the sets \(\calE_Z\) defined in Definition~\ref{defn:calEZ}. We leverage these results in Section~\ref{sec:ProofOfConsequenceZP} to prove the general Theorem~\ref{thm:ConsequenceZP_general}, which implies Theorem~\ref{thm:ConsequenceZPDistinctIsogeny} as a special case. Finally, in Section~\ref{sec:RelationZP}, we show that the Zilber--Pink conjecture (Conjecture~\ref{conj:ZilberPink}) implies Conjecture~\ref{conj:GeneralizedMM}.

%%%%%%%%%%%%%%%%%%%%%%%%%%%%%%%%%%%%%%%%%%%%%%%%%%%%%%%%%%%%%%
\subsection{The case when \(Z\) is a nontrivial torsion translate}\label{sec:ZPTorsion}
\begin{lemma}\label{lem:TorsionTranslate}
    Let \(B \subset A\) be an abelian subvariety and let \(\tau \in A^r_{\tors}\).  Then
    \[
         M_{\tau + B^r}:=\{\be \in \Z^r : (\tau + B^r) \cap  A^{r}[\be] \neq \emptyset\}
    \]
    is a full rank subgroup of \(\Z^r\), which is a proper subgroup if and only if \(\tau \not\in B^r\).  Explicitly, when \(r=1\), we have \(M_{\tau + B^r} = \ord_{A/B}(\tau)\Z\).
\end{lemma}

\begin{proof}
    By definition, \(\be \in M_{\tau + B^r}\) if and only if \([\be] (\tau + B^r) = [\be]\tau + B\) contains \(0 \in A\).  Hence \(\be \in M_{\tau + B^r}\) if and only if \([\be]\tau \in B\).  The set \(M_{\tau+B^r}\) is thus a subgroup.
    Since \(\tau\in A_{\tors}^r\), for each \(1\leq i\leq r\), there exists a positive integer \(m_i\) such that \(m_i\textup{pr}_i(\tau) \in B\) (here \(\textup{pr}_i\) denotes the \(i^{th}\) projection). In particular, for each \(1\leq i\leq r\), the vector
    \((0, \dots, 0, m_i, 0, \dots, 0)\) is contained in \(M_{\tau + B^r}\).  We conclude that the subgroup has full rank. If \(\tau \not\in B^r\), we know that some \(\textup{pr}_i(\tau) \not\in B\), and so \((0, \dots,0, 1, 0, \dots, 0) \not\in M_{\tau + B^r}\).  Thus, in this case, this subgroup is not all of \(\Z^r\). 

    The explicit description when \(r=1\) follows from the definition \(M_{\tau + B^r} = \{e \in \Z : e\tau \in B\}\).
\end{proof}

%%%%%%%%%%%%%%%%%%%%%%%%%%%%%%%%%%%%%%%%%%%%%%%%%%%%%%%%%%%%%%
\subsection{Consequence of the Manin--Mumford Conjecture}\label{sec:ProofOfConsequenceMM}
\begin{theorem*}[Manin--Mumford Conjecture, proved by Raynaud~\citelist{\cite{Raynaud-MMcurves}*{{Th\'eor\`eme} I} \cite{Raynaud-MMgeneral}*{Th\'eor\`eme, page 327} }]
    Let \(Z\subset A\) be a closed subvariety of an abelian variety. Then
    \[
       {Z\cap A_{\textup{tors}}} = \bigcup_{i=1}^s \tau_i + (B_i)_{\tors},
    \]
    where for all \(i\), \(\tau_i\in A_{\textup{tors}}\) and \(B_i < A\) is an abelian subvariety such that \(\tau_i + B_i\subset Z\).
\end{theorem*}

\begin{proof}[Proof of Proposition~\ref{prop:consequenceMM}]
    By treating each irreducible component separately (and taking the (finite) union of the corresponding lists of integers) the statement is equivalent to the case when \(Z\) is irreducible.
    We will prove this by induction on the dimension of \(Z\).  
    If \({Z \cap A_{\tors}}\) is not Zariski dense in \(Z\), then, replacing \(Z\) with each component of the Zariski closure of \({Z \cap A_{\tors}}\), we are done by induction.
    If \({Z \cap A_{\tors}}\) is Zariski dense in  \(Z\), then, by the Manin--Mumford conjecture, \(Z\) is a torsion translate \(Z = \tau + B\) of an abelian subvariety \(B < A\).  Since \(0 \notin Z\), we have \(\tau \not\in B\) and so \(\ord_{A/B}(\tau) > 1\). The result therefore follows from Lemma~\ref{lem:TorsionTranslate}, letting \(m_1 = \ord_{A/B}(\tau)\).
\end{proof}

\subsection{The structure of \(\calE_Z\)}\label{sec:EZ}

We begin by refining the definition of the exceptional set \(\calE_Z\).

\begin{defn}\label{defn:EZB}
    Given an abelian subvariety \(B \subset A\), we define
    \[\calE_Z^B \colonequals \{ \be \in \Z^r : [\be]\langle Z \rangle^\circ \subseteq B\}.\]
\end{defn}

It is immediate from \eqref{eq:calEZ_alternate} in Definition~\ref{defn:calEZ} that \(\calE_Z = \bigcup_{B \subset A \atop \dim B \leq \dim Z} \calE_Z^B\).  If \(A\) is simple and \(\dim Z  < \dim A\), then \(\calE_Z = \calE_Z^{\{0\}}\).

Suppose that \(Z \subset A^r_{\tors}+ B^r\) for some abelian subvariety \(B \subset A\).  Then \(\langle Z \rangle^\circ \subset B^r\) and so \([\be]\langle Z \rangle^\circ \subset [\be]B^r \subset B\) for all \(\be \in \Z^r\).  Consequently, in this case, \(\calE_Z^B = \Z^r\).  The following shows that outside of this case, \(\calE_Z^B\) is a subgroup of smaller rank.

\begin{lemma}\label{lem:EZB_subgroup}
    Let \(A\) be an abelian variety, let \(B \subset A\) be an abelian subvariety and let \(Z\subset A^r\) be an integral subvariety.  Then \(\calE_Z^B \subset \Z^r\) is a  subgroup. 
    If \(Z\) is not contained in \(A^r_{\tors} + B^r\), then \(\calE_Z^B\) has rank strictly less than \(r\).
\end{lemma}
\begin{proof}
    If \(\be_1, \be_2 \in \calE_Z^B\), then, since the map \([\be_1 + \be_2] \colon A^r \to A\) factors as
    \[A^r \xrightarrow{[\be_1] \times [\be_2]} A \times A \xrightarrow{+} A,\]
    we have
    \[[\be_1 + \be_2] \langle Z \rangle^\circ \subset [\be_1]\langle Z \rangle^\circ + [\be_2]\langle Z \rangle^\circ \subset B + B = B.\]
    Thus \(\be_1 + \be_2 \in \calE_Z^B\).  Similarly, we have \(n \be_1 \in \calE_Z^B\) for any \(n \in \Z\).  Hence \(\calE_Z^B\) is a subgroup of \(\Z^r\).  
    
    If \(\be \in \calE_Z^B\), then \([\be]\langle Z \rangle^\circ \mod{B}\) is zero in \(A/B\).  Hence \(\langle Z \rangle^\circ \mod B^r\) is contained in \((A/B)^r[\be]\).
    If \(\calE_Z^B\) has rank \(r\), then, by Lemma~\ref{lem:r_lin_independent} below, \(\langle Z \rangle^\circ \mod B^r \subset (A/B)^r_{\tors}\), and so \(\langle Z \rangle^\circ  \subset A^r_{\tors} + B^r\).  Since every component of \(\langle Z \rangle\) is a torsion translate of the identity component \(\langle Z \rangle^\circ\), we have \(\langle Z \rangle  \subset A^r_{\tors} + B^r\).  Consequently, \(Z \subset \langle Z \rangle \subset A^r_{\tors} + B^r\).
    \end{proof}

\begin{lemma}\label{lem:r_lin_independent}
    For \(\Q\)-linearly independent \(\be_1, \dots, \be_r \in \Z^r\), the intersection \(A^r[\be_1]\cap  \dots \cap A^r[\be_r]\) is contained in \(A^r_{\tors}\) (and, in particular, is \(0\)-dimensional).
\end{lemma}
\begin{proof}
    Let \(E \in M_r(\Z)\) be the matrix whose rows are \(\be_1, \dots, \be_r\).  Then \((z_1, \dots, z_r) \in A^r[\be_1]\cap A^r[\be_2]\cap \dots \cap A^r[\be_r]\) if and only if \(E (z_1, \dots, z_r)^T = \mathbf{0}.\) By assumption \(E\) is a matrix with full rank, so \(\det(E)\neq 0\). If \(E(z_1, \dots, z_r)^T = \mathbf{0}\), then 
    \[(\textup{Adj}(E)E)(z_1, \dots, z_r)^T =  (\det(E)z_1, \dots, \det(E)z_r)^T = \mathbf{0},\] 
    and so \(A^r[\be_1]\cap A^r[\be_2]\cap \dots \cap A^r[\be_r] \subset A^r_{\tors}\) is \(0\)-dimensional. 
\end{proof}

\begin{cor}\label{cor:EZ_eq_Zr}
    Let \(A\) be an abelian variety with distinct isogeny factors and let \(Z \subset A^r\) be an integral subvariety.  Then \(\calE_Z  = \Z^r\) if and only if there exists an abelian subvariety \(G \subset A\) with \(\dim G \leq \dim Z\) such that \(Z \subset A^r_{\tors} + G^r\).
\end{cor}
\begin{proof}
    Since \(A\) has distinct isogeny factors, it contains only finitely many abelian subvarieties. Let \(G_1, \dots, G_s \subset A\) be the abelian subvarieties with \(\dim G_i \leq \dim Z\). We therefore have \(\calE_Z = \bigcup_{i=1}^s \calE^{G_i}_Z\). By Lemma~\ref{lem:EZB_subgroup},  \(\calE_Z^{G_i}\) is a subgroup of rank at most \(r-1\) unless \(Z \subset A^r_{\tors} + G_i^r\).  Since a finite union of subgroups of rank at most \(r-1\) cannot equal \(\Z^r\), if \(\calE_Z = \Z^r\), then \(Z \subset A^r_{\tors} + G_i^r\) for some \(i\).  Conversely, if \(Z \subset A^r_{\tors} + G_i^r\), then \(\langle Z \rangle^\circ \subset G_i^r\) and so \([\be]\langle Z \rangle^\circ \subset [\be]G_i^r \subset G_i\) for all \(\be \in \Z^r\).  Hence, \(\calE_Z^{G_i} = \Z^r\).
\end{proof}

\begin{remark}
    Corollary~\ref{cor:EZ_eq_Zr} can fail if \(A\) does not have distinct isogeny factors, as the following example shows.
\end{remark}

\begin{example}
    Let \(E\) be an elliptic curve and let \(A = E^3\).  Consider the abelian subvariety \(H \simeq E^3\) of \(A^3\) defined by
    \[H \colonequals \{((p, q, 0), (r, 0, q), (0, -r, p)) : p, q, r \in E\}.\]
    For any \(\be = (e_1, e_2, e_3) \in \Z^3 \smallsetminus \{\mathbf{0}\}\), the image \([\be]H \subset A = E^3\) is contained in the subgroup \(E^3[(e_3, e_2, -e_1)]\).  Consequently, \(\dim [\be]H \leq 2\) for all \(\be \in \Z^3\).

    Let \(Z \subset H\) be a generic integral surface.  Hence \(\langle Z \rangle^\circ = H\) and so we have \(\calE_Z = \{\be \in \Z^3 : \dim [\be]H \leq 2\} = \Z^3\).  We will show that there is no abelian subvariety \(G \subset A\) with \(\dim G \leq 2\) such that \(Z \subset A^3_{\tors} + G^3\).  
    
    For any torsion point \(\tau\in A^r_{\tors}\), the algebraic groups generated by \(Z\) and \(Z - \tau\) have the same identity component. Thus, if \(Z\subset \tau + G^3\), then \(H = \langle Z \rangle^0 = \langle Z - \tau \rangle^0 \subset G^3\), which then implies that \(H_i\colonequals \pi_i(H)\subset G\) for all \(i=1,2,3\).
    Since each \(H_i\) is a distinct abelian surface in \(A=E^3\) (namely, the surface where the \((4-i)\)th coordinate is \(0\)), we have \(\langle H_1, H_2, H_3\rangle = A\).  As such, there is no proper abelian subvariety \(G \subset A\) such that \(H \subset G^3 \subset A^3\).
\end{example}

\begin{prop}\label{prop:Zariski_closed}
    Let \(A\) be an abelian variety and let \(Z \subset A^r\) be an integral subvariety.  If \(\calE_Z \neq \Z^r\), then there exists a proper Zariski closed subvariety \(V \subsetneq \C^r\) such that \(\calE_Z \subset V \cap \Z^r\).
\end{prop}
\begin{proof}
    We have \(A(\C) \simeq \C^g / \Lambda\), where \(\Lambda \subset \C^g\) is a lattice of rank \(2g\).  The abelian subvariety \(\langle Z \rangle^\circ \subset A^r\) corresponds to a subspace \(S \subset (\C^{g})^r\) such that \(\rk(S \cap \Lambda^r) = 2 \dim S\).  We will represent elements of \((\C^g)^r\) as \(r \times g\) matrices. The map of abelian varieties \([\be] \colon A^r \to A\) corresponds to the linear map \(\mat(r, g) \to \C^g\) given by multiplication by the row matrix \([\be]\) on the left.
    
    The condition that \(\dim [\be]\langle Z \rangle^\circ \leq \dim Z\) is equivalent to the condition that the image subspace \([\be]S\) has dimension at most \(\dim Z\). Let \(s_1, \dots, s_k\) be a basis of \(S\).  Then \(\dim [\be]S \leq \dim Z\) if and only if all of the \((\dim Z+1)\)-minors of the matrix with rows \([\be]s_1, \dots, [\be]s_k\) vanish.  This either holds identically for all \(\be \in \C^r\), or the \(\be\) for which it holds are contained in proper Zariski closed \(V \subset \C^r\).
\end{proof}

\begin{example}
    Let \(B\) be a simple abelian variety and let \(A = B^3\). We will take \(Z\subset A^3\) to be a generic hypersurface in the subgroup \(H\subset A^3\) that is the image of
    \[
        B^3 \stackrel{\psi}{\hookrightarrow} A^3 = B^9, (b_1,b_2,b_3) \mapsto (b_2,b_1 + b_2,b_3)\times (b_1,0, b_3)\times (b_3, b_2, b_1 + b_3).
    \]
    We claim that for such \(Z\) the conclusion of Proposition~\ref{prop:Zariski_closed} is as strong as possible.

    From the definition of \(\calE_Z\) and the fact that every proper algebraic subgroup of \(A\) has dimension at most \(2\dim B\), we have the following sequence of equalities:
    \[
    \calE_Z = \left\{\be : \dim Z = \dim A - 1 {\geq} \dim [\be]\langle Z\rangle\right\} = \left\{\be : \dim [\be]H \leq 2\dim B\right\} = \left\{\be :  [\be]H \subsetneq A\right\}
    \]

    The subgroup \([\be] H\) is the image of the composition \(B^3\xrightarrow{\psi}A^3 \xrightarrow{[\be]} B^3\), which can be represented by the matrix equation
    \[
    \begin{bmatrix} b_1\\b_2\\b_3\end{bmatrix} \mapsto 
    e_1\begin{bmatrix} b_2\\b_1 + b_2\\b_3\end{bmatrix} + 
    e_2\begin{bmatrix} b_1\\0\\b_3\end{bmatrix} + 
    e_3\begin{bmatrix} b_3\\b_2\\b_1 + b_3\end{bmatrix} =
    \begin{bmatrix} 
    e_2 & e_1 & e_3\\
    e_1 & e_1 + e_3 & 0\\
    e_3 & 0 & e_1 +e_2 + e_3
    \end{bmatrix}\begin{bmatrix} b_1\\b_2\\b_3\end{bmatrix}.
    \]
    In particular, \([\be]H \subsetneq A\) if and only if 
    \[
    \det\left(    \begin{bmatrix} 
    e_2 & e_1 & e_3\\
    e_1 & e_1 + e_3 & 0\\
    e_3 & 0 & e_1 +e_2 + e_3
    \end{bmatrix}\right) = (e_2)(e_1 + e_3)(e_1 + e_2 + e_3) - e_1^2(e_1 + e_2 + e_3) - e_3^2(e_1 + e_3) =0,
    \]
    so set of such \(\mathbf{e}\) are the integral points on the cone over a genus \(1\) curve \(C_0\). This curve \(C_0\) also evidently has rational points (set \(e_3 = 0\) and solve) and so is isomorphic to its Jacobian. A \texttt{Magma}~\cite{Magma} computation shows that the Jacobian has short Weierstrass equation \(y^2 = x^3 - 2x\) and has infinitely many rational points~\cite{LMFDB}*{\href{https://www.lmfdb.org/EllipticCurve/Q/256/b/1}{256b1}}. Thus, \(\calE_Z\) cannot be contained in a finite union of proper subgroups of rank at most \(2\), but is equal to the integral points on a Zariski closed subvariety.
\end{example}

\subsection{The general case}\label{sec:ProofOfConsequenceZP}

\begin{thm}\label{thm:ConsequenceZP_general}
    Let \(A\) be an abelian variety and let \(Z\subset A^r\) be an integral subvariety such that \(\dim Z < \dim A\). 
Assume Conjecture~\ref{conj:GeneralizedMM}.  
    \begin{enumerate}
    \item \label{part:ConsZP_general_1} There exists a proper Zariski closed \(V \subset \C^r\) and torsion translates \(\tau_1 + H_1, \dots, \tau_n + H_n\) satisfying 
    \begin{enumerate}
        \item \(\dim Z \cap (\tau_j  + H_j) \geq \dim [\be]H_j\) for all \(\be \in \Z^r\) and all \(1 \leq j \leq n\),
        \item \( Z \cap \bigcup_{\be\in \Z^r \smallsetminus V} A^r[\be] \subset \bigcup_{j=1}^n \tau_j  + H_j\).
    \end{enumerate}
    If, furthermore, \(A\) has distinct isogeny factors, then \(V\) may be chosen so that \(V \cap \Z^r\) is a union of subgroups \(M_1, \dots, M_t\) of rank at most \(r-1\) and \(H_j = B_j^{r}\) for some \(B_j\) an abelian subvariety of \(A\).
    \item\label{part:ConsZP_general_2} Suppose that \(A\) has distinct isogeny factors, which implies that \(H_j = B_j^r\).  If there exists an \(\tilde\be\in \Z^r\smallsetminus\{0\}\) such that \([\tilde{\be}]\tau_j \not\in B_j\) for all \(1 \leq j \leq n\), then there exist rank \(r\) proper subgroups \(M_{t+1}, \dots, M_{t+n}\) with \(\tilde{\be} \not\in \cup_{i=t+1}^{{t+n}} M_i\) (which implies \(\cup_{i=1}^{t+n} M_i\subsetneq \Z^r\)) such that
 \[
     Z \cap \bigcup_{\be\notin \cup_{i=1}^{{t+n}} M_i} A^r[\be] = \emptyset.
 \]
\end{enumerate}
\end{thm}

\begin{proof}
    We will prove the first statement by induction on the dimension of \(Z\).  The base case is when \(Z = \emptyset\), in which case the statement holds with \(V = \emptyset\) and \(n=0\). 
     
     Now assume that \(Z\) is nonempty. We first consider the case that \(\calE_Z = \Z^r\). By definition, for all \(\be \in \Z^r\), we have \(\dim Z \geq \dim [\be]\langle Z \rangle^\circ\). Let \(H \colonequals \langle Z \rangle^\circ\). Since \(Z\) is irreducible, it lies in a unique irreducible component of \(\langle Z \rangle\), which is a torsion translate \(\tau + H\).  Hence the conclusion holds with \(V = \emptyset\), \(n= 1\), \(\tau_1 = \tau\), and \(H_1 = H\).  If \(A\) has distinct isogeny factors and \(\dim Z < \dim A\), then
     by Corollary~\ref{cor:EZ_eq_Zr}, there exists a proper subvariety \(G \subset A\) and torsion point \(\tau\in A^r_{\tors}\) such that \(Z \subset \tau + G^r\). 
     For every \(\be \in \Z^r\) we have 
     \[\dim Z \cap (\tau + G^r) = \dim Z \geq \dim G \geq \dim [\be]G^r.\]
     Hence the conclusion holds with \(V = \emptyset\), \(n=1\), \(\tau_1 = \tau\) and \(H_1 = G^r\).

     It remains to consider the case that \(\calE_Z \neq \Z^r\).  By Proposition~\ref{prop:Zariski_closed}, there exists a proper Zariski closed \(V_0 \subset \C^r\) such that \(\calE_Z \subset V_0 \cap \Z^r\).  In the case that \(A\) has distinct isogeny factors, for all abelian subvarieties \(G_1, \dots, G_s \subset A\) with \(\dim G_i \leq \dim Z\), we have \(Z \not\subset A^r_{\tors} + G_i^r\) by Corollary~\ref{cor:EZ_eq_Zr}.  Hence each \(\calE_Z^{G_i}\) is a subgroup of rank less than \(r\) by Lemma~\ref{lem:EZB_subgroup}.  In this case, we may therefore take \(V_0\) to be such a union of smaller-dimensional linear spaces.
     
     By Conjecture~\ref{conj:GeneralizedMM},
     \[
        Z\cap \bigcup_{\be\in \Z^r\smallsetminus V_0}A^r[\be]
     \]
     is not Zariski dense in \(Z\). Thus, there exist \(W_1, \dots, W_m\subset Z\), each integral and of strictly smaller dimension than \(Z\) such that
     \begin{equation}\label{eq:components}
        Z\cap \bigcup_{\be\in \Z^r \smallsetminus V_0}A^r[\be]
        = \bigcup_{\ell=1}^m \left(W_\ell\cap \bigcup_{\be\in \Z^r \smallsetminus V_0}A^r[\be]\right).
          \end{equation}
     For fixed \(1 \leq \ell \leq m\), by induction, there exists a proper Zariski closed subset \(V_{\ell} \subset \C^r\) and torsion translates \(\tau_{\ell, 1} + H_{\ell, 1}, \dots, \tau_{\ell, n_\ell} + H_{\ell, n_\ell}\) satisfying \(\dim W_\ell \cap (\tau_{\ell, j} + H_{\ell, j}) \geq \dim [\be] H_{\ell, j}\) for all \(\be\) and \(1 \leq j \leq n_\ell\) such that
     \begin{equation}\label{eq:W}
     W_\ell\cap \bigcup_{\be\in \Z^r \smallsetminus V_\ell}A^r[\be] \subset \bigcup_{j=1}^{n_\ell} \tau_{\ell, j} + H_{\ell, j}.
     \end{equation}
     Further, when \(A\) has distinct isogeny factors, we may assume that the \(V_\ell\) and \(H_{\ell, j}\) have the claimed special structure.
    Since \(W_\ell \subset Z\), we  have \(\dim Z\cap (\tau_{\ell, j} + H_{\ell, j}) \geq \dim W_\ell\cap (\tau_{\ell, j} + H_{\ell, j}) \geq \dim [\be] H_{\ell, j}\) for all \(j\).  
    Combining \eqref{eq:components} with \eqref{eq:W} for each \(\ell\), the desired conclusion therefore holds with the proper Zariski closed subset \(V \colonequals V_0 \cup \bigcup_{\ell=1}^m V_\ell\), and torsion translates \(\bigcup_{\ell=1}^m \bigcup_{j=1}^{n_\ell} \tau_{\ell, j} + H_{\ell, j}\).

    It remains to prove the second statement. Assume that there exists a nonzero \(\tilde\be\) such that \([\tilde{\be}]\tau_i \not\in B_i\) for all \(1 \leq i \leq n\). Let \(M_{t+i}=M_{\tau_i + B_i^r}\) for \(1 \leq i \leq n\) be the subgroup from Lemma~\ref{lem:TorsionTranslate} with the property that \((\tau_i + B_i^r)\cap A^r[\be] \neq \emptyset\) if and only if \(\be \in M_{\tau_i + B_i^r}\).  By our assumption, \(\tilde{\be} \not\in \bigcup_{i=1}^n M_{t+i}\), and, in particular, each of these \(M_{t+i}\) is a proper rank \(r\) subgroup of \(\Z^r\). 
    A finite union of proper subgroups of \(\Z^r\) that excludes a single element must exclude an entire coset of a finite index subgroup~\cite{VirayVogt}*{Lemma 4.3.8}. Therefore the proper containment \(\bigcup_{i=t+1}^{t+n}M_i\subsetneq \Z^r\) implies that \(\bigcup_{i=1}^{t+n}M_i\) is also properly contained in \(\Z^r\).
    Finally, by part~\eqref{part:ConsZP_general_1}, 
    \[Z\cap \bigcup_{\be\notin \cup_{i=1}^{t+n} M_i} A^r[\be] \subset \left(\bigcup_{i=1}^n \tau_i  + B^r_i\right) \cap \bigcup_{\be\notin \cup_{i=1}^{t+n} M_i} A^r[\be] = \emptyset.\qedhere\]
\end{proof}

\begin{proof}[Proof of Theorem~\ref{thm:ConsequenceZPDistinctIsogeny}]
    Taking \(B = A\) and \(\tau = 0\), we see that \(\dim Z < \dim A\), and therefore we may
    apply Theorem~\ref{thm:ConsequenceZP_general}.  Observe that for any torsion translate \(\tau_j + B_j^r\) appearing in part~\eqref{part:ConsZP_general_1}, we have \(\dim Z \cap (\tau_j + B_j^r) \geq \dim B_j\), since \(\dim[\be]B_j^r = \dim B_j\) for \(\be\neq \mathbf{0}\), and \(\dim Z \cap (\tau_j + B_j^r) \geq \dim[\be]B_j^r\) must hold for all \(\be \in \Z^r\).

    Let \(\tilde{\be}\in \Z^r\) be the given tuple. If \(\tilde{\be} = 0\), then since \([\mathbf{0}]\tau\) is contained in every abelian subvariety \(B \subset A\), we must have that \(\dim Z \cap (\tau + B^r) < \dim B\) for every \(B \subset A\).  In particular, we must have \(n=0\) in Theorem~\ref{thm:ConsequenceZP_general}\eqref{part:ConsZP_general_1} and so the result holds with \(s=t\) and all of the subgroups having rank less than \(r\).  If \(\tilde{\be} \neq 0\), then this \(\tilde{\be}\) satisfies Theorem~\ref{thm:ConsequenceZP_general}\eqref{part:ConsZP_general_2}, since, as we have observed \(\dim Z \cap (\tau_j + B_j^r) \geq \dim B_j\) for all \(1 \leq j \leq n\).  The result now follows.
\end{proof}

%%%%%%%%%%%%%%%%%%%%%%%%%%%%%%%%%%%%%%%%%%%%%%%%%%%%%%%%%%%%%%%
\subsection{Relationship to the Zilber--Pink conjecture}\label{sec:RelationZP}

\begin{conj}[Zilber--Pink]\label{conj:ZilberPink}
    Let \(B\) be an abelian variety.  Let \(V \subset B\) be a subvariety that is not contained in a proper algebraic subgroup of \(B\).  Then the intersection
    \[V \cap \left(\bigcup_{\substack{H \subset B \textup{ alg.~subgroup} \\ \dim H + \dim V < \dim B}} H \right) \]
    is not Zariski dense in \(V\).
\end{conj}

Given an algebraic group \(H\), we write \(H^\circ\) for the connected component of the identity.

\begin{lemma}\label{lem:identity_dim}
    Let \(H_1, H_2 \subset G\) be algebraic subgroups of an algebraic group.  Then we have \(\dim H_1 \cap H_2 = \dim H_1^\circ \cap H_2\).
\end{lemma}
\begin{proof}
    Any component of the intersection \(H_1 \cap H_2\) is a torsor under the identity component \((H_1 \cap H_2)^\circ\), and so \(\dim H_1 \cap H_2  = \dim (H_1 \cap H_2)^\circ\).  The result follows from the chain of inclusions 
    \[(H_1 \cap H_2)^\circ \subset H_1^\circ \cap H_2^\circ \subset H_1^\circ \cap H_2 \subset H_1 \cap H_2.\qedhere\]
\end{proof}

\begin{prop}
    Conjecture~\ref{conj:ZilberPink} implies Conjecture~\ref{conj:GeneralizedMM}.
\end{prop}
\begin{proof}
Since \(Z\) is integral, it is contained in one translate \(\tau + \langle Z \rangle^\circ\) of the identity component \(\langle Z \rangle^\circ\) of the group it generates, for some (possibly trivial) torsion point \(\tau\).  Hence the translate \(-\tau + Z\) is contained in \(\langle Z \rangle^\circ\).

We will apply Conjecture~\ref{conj:ZilberPink} to the abelian variety \(\langle Z \rangle^\circ\) and subvariety \(V \colonequals -\tau + Z\).  Let \(H_{\be} \colonequals  \langle Z \rangle^\circ \cap \langle-\tau + A^r[\be]\rangle\).
Applying Lemma~\ref{lem:identity_dim} twice, using the fact that \(\tau\) is a torsion point, we have 
\[
\dim \langle Z \rangle \cap A^r[\be] = \dim \langle Z \rangle^\circ \cap A^r[\be] = \dim \langle Z \rangle^\circ \cap \langle -\tau + A^r[\be]\rangle = \dim H_{\be}.
\]
Since \(\dim \langle Z \rangle^\circ = \dim \langle Z \rangle\) and \(\dim V  = \dim Z\), we have
\[\calE_Z  = \{\be : \dim \langle Z \rangle \cap A^r[\be] \geq \dim \langle Z \rangle - \dim Z\} = \{\be : \dim H_{\be} \geq \dim \langle Z \rangle^\circ - \dim V\}. \]
Thus, if \(\be \in \Z^r \smallsetminus \calE_Z\), then \(\dim H_{\be} + \dim V < \dim \langle Z \rangle^\circ\).  Hence Conjecture~\ref{conj:ZilberPink} implies that
\[V \cap \left( \bigcup_{\be \in \Z^r \smallsetminus \calE_Z} H_{\be} \right) \]
is not Zariski dense in \(V\).

Considering translation, we have
\[-\tau + (Z \cap A^r[\be]) = (-\tau + Z) \cap (-\tau + A^r[\be]) \subset (-\tau + Z) \cap \langle -\tau + A^r[\be] \rangle = V \cap H_{\be},\]
and hence 
\[Z \cap \left( \bigcup_{\be \in \Z^r \smallsetminus \calE_Z} A^r[\be] \right) \subset \tau  + \left(V \cap \left( \bigcup_{\be \in \Z^r \smallsetminus \calE_Z} H_{\be} \right)\right) \]
is not Zariski dense in \(\tau + V  = Z\), as desired.
\end{proof}

%%%%%%%%%%%%%%%%%%%%%%%%%%%%%%%%%%%%%%%%%%%%%%%%%%%%%%%%%%%%%%%%%%%%%%%%%%%%%%%%%%%%%%%%%%%%%%%%%%%%%%%%%%
%%%%%%%%%%%%%%%%%%%%%%%%%%%%%%%%%%%%%%%%%%%%%%%%%%%%%%%%%%%%%%%%%%%%%%%%%%%%%%%%%%%%%%%%%%%%%%%%%%%%%%%%%%
\section{Relating unlikely splitting profiles to unlikely intersections}\label{sec:UnlikelySplittingProfiles}

   The goal of this section is to show how the unlikely intersections results from Section~\ref{sec:MMConsequences} imply the results on unlikely splitting profiles from the introduction. We first give the general framework and then in Section~\ref{sec:ProofsOfMainTheorems}, we explain how to deduce the specific statements in the introduction.

    \subsection{The general framework}\label{sec:overview} The main focus of this paper is determining for which places \(v\) a given splitting profile \(\spp\) is realized by points \(x\in C\) with \([x]\in \langle [D]\rangle+H\), for some \(k\)-rational divisor \(D\) and abelian subvariety \(H\subset J_C\). The applications will all rely on the case when \(H = \{0\}\), but we give the framework for arbitrary \(H\) with a view to future exploration of unlikely splitting profiles in the case when \(J_C\) is positive rank.
    
    For \(v\) a nonarchimedean place of good reduction, we can spread out and make use of reduction modulo \(v\). The first step in our approach is to observe that  existence of \textit{global} points with prescribed ramification at \(v\) implies the existence of \(\F_v\)-points in a certain subvariety of a product of symmetric products.
    \begin{lemma}\label{lem:SplProfileImpliesFvPoints}
        Let \(C\) be a nice curve over a number field \(k\), let \(v\) be a prime of good reduction, let \(D\) be a \(k\)-rational divisor of positive degree on \(C\), let \(x\in C\) be a closed point, and let \(H\subset J_C\) be an abelian subvariety. Write \(\spp:= \spp_v(\kk(x)/k), r := r(\spp), \bff := \bff(\spp), \be := \be(\spp)\). If \([x]\in \langle [D]\rangle + H\), then there exist points \(\overline{x}_1, \dots, \overline{x}_{r}\) on \(C_{\F_v}\) with \(\deg(\overline{x}_i) \mid f_i\) such that
        \[
        \sum_i e_i\cdot\textstyle\frac{f_i}{\deg(\bar{x}_i)}[\overline{x}_i] \in \textstyle{\frac{\deg(x)}{\deg(D)}}\left[D_{\F_v}\right] + H_{\F_v}. 
        \]
        In particular, there is an \(\F_v\)-point of \(\Sym^{\bff}_{C_{\F_v}}\) that maps to \(\frac{\deg(x)}{\deg(D)}\left[D_{\F_v}\right] + H_{\F_v}\) under the composition
        \begin{equation}\label{eq:SymComposition}
                \Sym^{\bff}_C\xrightarrow{\textup{Abel--Jacobi}} W^{\bff}_C \xrightarrow{[\mathbf{e}]} \Pic^{\deg(\spp)}_C.
        \end{equation}

        Further, given an integral model \(\calC \to \Spec \OO_{k}\), if the map \(\Spec \OO_{\kk(x)} \to \calC\), corresponding to the point \(x\), is injective on geometric points over \(v \in \Spec \OO_{k}\) (which occurs for all but finitely many primes), then \(\deg(\bar{x}_i) = f_i\) and \(\bar{x}_i \neq \bar{x}_j\) for \(i \neq j\).
    \end{lemma}
    \begin{proof}
        Let \(\calC \to \Spec \OO_{k,S}\) be a smooth proper model of \(C/k\).  The closed point \(x\) spreads out to a map \(\Spec \OO_{\kk(x), S'} \to \calC\), where \(S'\) is the set of places of \(\kk(x)\) above \(S\).  By definition of the splitting profile, the fiber of \(\Spec \OO_{\kk(x), S'}\) over \(v \in \Spec \OO_{k,S}\) is \(r\) distinct points of degree \(f_i\) over \(\Spec \F_v\) and multiplicity \(e_i\).  The image of each of these points is a closed point \(\bar{x}_i \in C_{\F_v}\) of degree dividing \(f_i\).  Hence, the fiber of the image of \(\Spec \OO_{\kk(x), S'} \to \calC\) over \(v\) is equal (as a divisor) to \(\sum_i e_i\frac{f_i}{\deg(\bar{x}_i)}\overline{x}_i\).  If the map \(\Spec \OO_{\kk(x), S'} \to \calC\) is injective on geometric points over \(v\), then \(\deg(\bar{x}_i) = f_i\) and \(\bar{x}_i \neq \bar{x}_j\) for \(i \neq j\).
        
        Since \([x]\in \langle [D]\rangle + H\) is a closed condition, it also holds modulo \(v\). In other words, the \(\F_v\)-point \(\left(\frac{f_i}{\deg(x_i)}\overline{x}_{i}\right)_i\in \Sym^{\bff}_{C_{\F_v}}\) maps to \(\frac{\deg(x)}{\deg(D)}\left[D_{\F_v}\right] + H_{\F_v}.\)
    \end{proof}
     Since the morphism~\eqref{eq:SymComposition} and the subvariety \(\frac{\deg(x)}{\deg(D)}[D] + H\) are defined over \(k\), this gives a geometric criterion to determine whether a fixed splitting profile \(\spp\) is likely in \(|\frac{\deg(\spp)}{\deg(D)} D + H|\).

     \begin{lemma}\label{lem:SymIntersectionEmptyImpliesUnlikely}
        Let \(C\) be a nice curve over a number field \(k\), let \(D\) be a \(k\)-rational divisor of positive degree on \(C\), and let \(\spp\) be a splitting profile with \(\deg(D)|\deg(\spp)\). 
        Let \(H \subset J_C\) be an abelian subvariety.
        If the image of the composition
        \[
        \Sym^{\bff}_C\xrightarrow{\textup{Abel--Jacobi}} W^{\bff}_C \xrightarrow{[\mathbf{e}]} \Pic^{\deg(\spp)}_C,
        \]
        does not intersect \([\frac{\deg(\spp)}{\deg(D)}D] + H\), then the splitting profile \(\spp\) is unlikely in \(|\frac{\deg(\spp)}{\deg(D)} D + H|\).   
    \end{lemma}

     \begin{proof}
        For any finite set of places \(S\) and variety \(\calV \to \Spec \OO_{k,S}\), the set of places \(v\notin S\) where \(\calV\) has an \(\F_v\)-point is a subset of the support \(\Supp(\calV)\subset \Spec\OO_{k,S}\). In particular, if \(\calV_k = \emptyset\), then there are finitely many places \(v\) where \(\calV\) has an \(\F_v\)-point. Now we apply this for \(\calV\) an integral model of 
        the complete preimage of \(\frac{\deg(\spp)}{\deg(D)} [D] + H\) under the composition of maps~\eqref{eq:SymComposition}. 
     \end{proof}
    The geometric criterion in Lemma~\ref{lem:SymIntersectionEmptyImpliesUnlikely} is most useful when one is interested in a finite set of possible profiles, each of manageable degree. Indeed, Lemmas~\ref{lem:SplProfileImpliesFvPoints} and~\ref{lem:SymIntersectionEmptyImpliesUnlikely} give us a method for computing the sets of primes realizing a given splitting profile, such as is given in Table~\ref{table:genus2}; see Section~\ref{sec:ProofOfTable} for more details. However, since the intersection depends on all aspects of the splitting profile, it is of limited use when considering an infinite set of profiles as discussed in the main theorems of the paper.

    One of the main insights in the paper is that for the (infinitely many) splitting profiles that share a common \(\bff\) tuple, the intersections under study can be placed in the common framework of the intersection of a fixed subvariety of \((J_C/H)^r\) with a (varying) algebraic subgroup.
    When \(\bff = (f)\) has length \(1\), we intersect with the torsion subgroups of \(J_C/H\).  In general, we intersect with the subgroups \(A^r[\be(\spp)]\) studied in Section~\ref{sec:MMConsequences},
    where \(A = J_C/H\).

    \begin{prop}\label{prop:JacIntersectionEmptyImpliesUnlikely}
        Let \(C\) be a nice curve over a number field \(k\), let \(D\) be a \(k\)-rational divisor of positive degree on \(C\), and let \(H\subset J_C\) be an abelian subvariety. Let \(\spp^0 = \{(e_i^0, f_i)\}\) be a splitting profile with \(\deg(D) \mid e^0_if_i\) for all \(i\) and write \(\be^0 \colonequals \be(\spp^0)\) and \(\bff \colonequals \bff(\spp^0)\). 
        For \(\be = (e_1, \dots, e_r) \in \Diag(\be^0)\Z^r_{>0}\), if the intersection of \(k\)-varieties
            \begin{equation}\label{eq:JacIntersection}
                \left(\pi_H^{r}\circ\iota_{D} \circ [\Diag(\be^0)] \right)\left(W^{\bff}_{C}\right)\cap (J_{C}/H)^r[\be/\be^0]
            \end{equation}
        is empty, then \(\spp \colonequals \{(e_i, f_i)\}\) is unlikely in \(|\frac{\deg(\spp)}{\deg(D)} D + H|\).
    \end{prop}
    \begin{remark}
        Any splitting profile \(\spp\) for which \(\deg(D) \mid ef\) for all \((e, f) \in \spp\) can be understood in the context of Proposition~\ref{prop:JacIntersectionEmptyImpliesUnlikely} with 
                 \[\spp^0 \colonequals \left\{\left(\frac{\deg (D)}{\gcd(f, \deg(D))}, f\right) : (e,f)\in \spp\right\}.\]
    \end{remark}

    \begin{proof}
        By definition of the subgroup \((J_{C}/H)^r[\be/\be^0]\), the intersection~\eqref{eq:JacIntersection} is empty if and only if the image of the composition
        \begin{equation}\label{eq:WfComposition}
            W^{\bff}_C\xrightarrow{[\Diag(\be^0)]}\prod_i W^{e_i^0f_i}_C \xrightarrow{\iota_D} J_C^{r} \xrightarrow{\pi_H^r} (J_C/H)^r \xrightarrow{[\be/\be^0]} J_C/H
        \end{equation}
        does not contain \(0\). This composition fits into the following commutative diagram.
        \[
        \begin{tikzcd}
            \Sym^{\bff}_C \arrow[r, two heads] & W^{\bff}_C \arrow[rr, "{[\Diag(\be^0)]}"] \arrow[rrd, "{[\be]}"] &&  \prod_i W^{e_i^0f_i}_C \arrow[r, "{\iota_D}"] \arrow[d, "{[\be/\be^0]}"] &J_C^r \arrow[r, "{\pi_H^r}"] \arrow[d, "{[\be/\be^0]}"] &(J_C/H)^r \ar[d, "{[\be/\be^0]}"]\\
             &  && \Pic^{\deg(\spp)}_C \arrow[r, "{\iota_D}"] & J_C \arrow[r, "{\pi_H}"] & J_C/H 
        \end{tikzcd}
        \]
        Since along the bottom row, the complete preimage of \(0\) is \([\frac{\deg(\spp)}{\deg(D)}D] + H\), the image of the composition~\eqref{eq:WfComposition} does not contain \(0\) if and only if the image of~\eqref{eq:SymComposition} does not meet \([\frac{\deg(\spp)}{\deg(D)}D] + H\). Thus, the assumptions of the proposition allow us to conclude by applying Lemma~\ref{lem:SymIntersectionEmptyImpliesUnlikely}.
    \end{proof}

    While Proposition~\ref{prop:JacIntersectionEmptyImpliesUnlikely}, may not appear to be substantially different than Lemma~\ref{lem:SymIntersectionEmptyImpliesUnlikely}, the fact that one component of the intersection~\eqref{eq:JacIntersection} depends only on \(\spp^0\) 
    and the other component is an algebraic subgroup means that the union over \(\be\in [\Diag(\be^0)]\Z^r\) of intersections~\eqref{eq:JacIntersection} falls under the framework of unlikely intersections (hence our choice of terminology).

    When \(r = 1\), we prove that there are a positive density set of degrees \(d\) where total ramification is unlikely in the linear system \(|dD|\) for \(D\) an ineffective divisor on a curve of genus at least \(2\). More generally, we use the Manin--Mumford conjecture to prove the following.

 \begin{theorem}\label{thm:unlikely_splitting_rank1}
    Let \(C\) be a nice curve over a number field \(k\), let \(D\in \Pic_C(k)\) be a divisor of positive degree, and let \(H\subset J_C\) be an abelian subvariety. Let \(\spp^0 = \{(e^0_i, f_i)\}\) be a splitting profile such that \(\deg(D) \mid e_i^0f_i\) for all \(i\) and write \(\be^0 \colonequals \be(\spp^0)\) and \(\bff \colonequals \bff(\spp^0)\).
    
    If \([\be^0](W^{\bff}_C)\cap \left(\frac{\deg(\spp^0)}{\deg(D)}[D] + H\right) = \emptyset\),
    then there exists \(m=m(C,\spp^0, D, H)\in \Z_{>0}\) such that 
    \[
(\pi_H\circ\iota_{D} \circ [\be^0])\left(W^{\bff}_C\right)\cap J_C/H[d]=\emptyset \quad \textup{for all }d\textup{ with }\gcd(d,m)=1.
\]  
Moreover, if \(m\) has the above property, then for all \(d\) coprime to \(m\), the splitting profile \(\mathsf{P} = \{(de^0_i, f_i)_{i=1}^r\}\) is \emph{unlikely} in \(\left|\frac{\deg(P)}{\deg(D)}[D]+H\right|\). 
\end{theorem}
\begin{remark}
    As we will see in the proof, the first statement about \(m\) is a straightforward consequence of the Manin--Mumford conjecture (Proposition~\ref{prop:consequenceMM}). The novelty is in the last statement.
\end{remark}
\begin{proof}[Proof of Theorem~\ref{thm:unlikely_splitting_rank1}]
    Since we assumed that \([\be^0](W^{\bff}_C)\cap \left(\frac{\deg(\spp^0)}{\deg(D)}[D] + H\right) = \emptyset\), we have
    \[
    0 \not\in Z \colonequals (\pi_H\circ \iota_{D} \circ [\be^0])\left(W^{\bff}_C\right)\subset J_C/H.
    \]  
    Since \(0 \not\in Z\), Proposition~\ref{prop:consequenceMM} implies that there exists a finite list of integers \(m_1, \dots ,m_s\in\Z_{\geq2}\) such that 
        \[
            Z\cap (J_C/H)[d] = \emptyset \quad \Longleftrightarrow \quad d\notin \bigcup_i m_i\Z.
        \]
    In particular, the first statement in Theorem~\ref{thm:unlikely_splitting_rank1} holds for any positive integer \(m\) divisible by \(\lcm(m_1, \cdots, m_s)\).

        Now take such an \(m\), and a positive integer \(d\) coprime to \(m\), which implies that \(Z\cap (J_C/H)[d] = \emptyset\). Let \(\be= \be(\spp) = (de^0_i)_{i=1}^r\). 
          Since \(\sigma\circ \pi_H^r \circ \iota_D\circ [\Diag(\be^0)] = \pi_H \circ \iota_D\circ[\be^0]\), the summation map \(\sigma\) sends \(({\pi_H^r \circ} \iota_{D} \circ [\Diag(\be^0)])(W^{\bff}_C)\cap (J_C/H)^r[\be/\be^0]\) to \(Z\cap (J_C/H)[d]\) and so
        \[
            (\pi^r_H \circ \iota_{D} \circ [\Diag(\be^0)])(W^{\bff}_C)\cap (J_C/H)^r[\be/\be^0] = \emptyset. 
        \]
        Thus, by Proposition~\ref{prop:JacIntersectionEmptyImpliesUnlikely}, the splitting profile \(\spp\) is unlikely in \(|\frac{\deg(\spp)}{\deg D}D + H|\). 
\end{proof}

    Theorem~\ref{thm:unlikely_splitting_rank1} gives a specific criterion, depending on a single splitting profile \(\spp^0\), to determine when unlikely splitting profiles arise when we allow the ramification degrees to vary over the rank \(1\) subgroup of \(\Z^r\) generated by \(\be(\spp^0)\). Restricting to a rank \(1\) subgroup of ramification degrees is a crucial assumption that allows us to leverage the Manin--Mumford conjecture.  The more general problem of studying the variation over high rank subgroups of \(\Z^r\) relates
    to intersections that arise in the (still open) Zilber--Pink conjecture. Conditional on the Zilber--Pink conjecture (or rather a special case given in Conjecture~\ref{conj:GeneralizedMM}), we obtain the following generalization of Theorem~\ref{thm:unlikely_splitting_rank1} to arbitrary splitting profiles.
 \begin{theorem}\label{thm:unlikely_splitting}
    Let \(C\) be a nice curve over a number field \(k\), let \(D\in \Pic_C(k)\) be a divisor of  positive degree and let \(H \subset J_C\) be an abelian subvariety.  Let \(\spp^{{0}} = \{(e_i^0, f_i)\}\) be a splitting profile such that \(\deg(D)|e^0_if_i\) for all \(i\). 
    
    Assume Conjecture~\ref{conj:GeneralizedMM}, that \(J_C\) has distinct isogeny factors, and that there exists an \(\tilde{\be}\in \Z^r\) such that for every abelian subvariety \(B\subset J_C/H\) and every torsion point \(\tau\in (J_C/H)^r\), either 
    \begin{itemize}
        \item \(\dim (\pi_H^r \circ \iota_D \circ [\Diag(\be^0)])(W^{\bff})\cap (\tau + B^r) < \dim B\); or 
        \item \([\tilde{\be}]\tau\notin B\). 
    \end{itemize}

     Then there exists a finite set of subgroups \(M_1, \dots, M_s\subsetneq \Z^r\) with \(\tilde{\be} \not \in M_i\) for every subgroup \(M_i\) of finite index (which implies that \(\bigcup_i M_i \subsetneq \Z^r\)) such that for each \(\be \in \Diag(\be^0)(\Z^r_{{>0}}\smallsetminus \bigcup M_i)\), the splitting profile \(\spp  = \{(e_i, f_i)_{i=1}^r\}\) is \emph{unlikely} in \(\left|\frac{\deg(\spp)}{\deg(D)}D + H\right|\).
\end{theorem}
\begin{proof}
    Applying Theorem~\ref{thm:ConsequenceZPDistinctIsogeny} to 
    \[Z \colonequals (\pi_H^r \circ \iota_D \circ [\Diag(\be^0)])(W^{\bff}) \subset (J_C/H)^r,\]
    there exist finitely many proper subgroups \(M_1, \dots, M_s \subset \Z^r\) with \(\bigcup_{i=1}^s M_i \neq \Z^r\) such that \(Z \cap \bigcup_{\be/\be^0 \not\in \cup M_i} (J_C/H)^r[\be/\be^0] = \emptyset\).  The conclusion follows from Proposition~\ref{prop:JacIntersectionEmptyImpliesUnlikely}.
\end{proof}

\subsection{Proofs of Theorems in the introduction}\label{sec:ProofsOfMainTheorems}
Theorems~\ref{thm:IntroTotallyRamified} and~\ref{thm:IntroRank1Splitting} are obtained as corollaries of Theorem~\ref{thm:unlikely_splitting_rank1} (Section~\ref{sec:CorOfUnlikely_rank1}). Theorem~\ref{thm:IntroUnlikelySplitting} is obtained as a corollary of Theorem~\ref{thm:unlikely_splitting} (Section~\ref{sec:CorOfUnlikelySplitting}). 

\subsubsection{Corollaries of Theorem~\ref{thm:unlikely_splitting_rank1}}\label{sec:CorOfUnlikely_rank1}
\begin{proof}[Proof of Theorem~\ref{thm:IntroTotallyRamified}]
    Since the Jacobian of \(C\) has rank \(0\), \(\Pic^i(C)\) is finite for every \(i \in \Z\). Set \(w \colonequals \textup{index}(C)\) and let \(\{[D_1], \dots, [D_n]\} = \Pic^w(C)\), where \(n = \#\Pic^0(C) = \#\Pic^w(C)\), since \(\Pic^w(C) \neq \emptyset\).

    By assumption, \([w](C) \subset \Pic^w_{C}\) does not contain \([D_j]\) for any \(1 \leq j \leq n\), since \(D_j\) is \(k\)-rational. Thus we may apply Theorem~\ref{thm:unlikely_splitting_rank1} with \(D = D_j, r = 1,\spp^0 = \{(w, 1)\}\) and \(H =0\). Then we obtain a positive integer \(m_j\) such that total ramification is unlikely in \(|dD_j|\) for all \(d\) coprime to \(m_j\).

    Let \(m = \lcm(m_1, \dots, m_n, {n})\). Let \(d\) be coprime to \(m\). Since \(d\) is coprime to \(n = \#\Pic^0(C)\), the multiplication-by-\(d\) map on \(\Pic^0\) is an isomorphism.  Consider a degree \(dw\) point \(x \in C\).  Then \([x] \in \Pic^{dw}(C)\), and so \([x - dD_1] \in \Pic^0(C)\) is of the form \(d[D_0]\) for some \([D_0] \in \Pic^0(C)\).  Thus \([x] = d[D_1 + D_0]\). Since \([D_1 + D_0] \in \Pic^w(C)\), we conclude that the degree \(dw\) points are contained in \(\cup_{j=1}^n |dD_j|\), and in each linear system total ramification is unlikely. Hence total ramification is unlikely for all degree \(dw\) points.    
\end{proof}
\begin{proof}[Proof of Theorem~\ref{thm:IntroRank1Splitting}]
    Set \(w \colonequals \textup{index}(C)\). Since the Jacobian of \(C\) has rank \(0\), \(\Pic^w(C)=\{[D_1], \dots, [D_n]\}\), where \(n = \#\Pic^0(C)\).

    Since \(\indx(C) \mid e^0_if_i\) by assumption, we have, \(\deg(D_j)|e^0_if_i\) for all \(j\) and \(i\). Further, since there are no points of degree at most \(\deg(\spp^0)\), then \([\frac{\deg(\spp^0)}{\deg(D_j)}D_j]\notin W^{\deg(\spp^0)}_C\) (and hence not in \([\be^0]W^{\bff}_C \subset W^{\deg(\spp^0)}_C\)). Thus we may apply Theorem~\ref{thm:unlikely_splitting_rank1} with \(D = D_j, H = 0\) to obtain a positive integer \(m_j\) such that the splitting profile \(\spp=\{(de^0_i, f_i)_{i=1}^r\}\) is unlikely in \(|d\frac{\deg(\spp^0)}{\deg(D_j)}D_j|\) for all \(d\) coprime to \(m_j\). 

    Let \(m \colonequals \lcm\left(m_1, \dots, m_n, {n}\right)\). Then if \(d\) is coprime to \(m\), it is coprime to \(m_j\) for all \(j\) and hence the splitting profile \(\spp=\{(de^0_i, f_i)_{i=1}^r\}\) is unlikely in \(|\frac{\deg(\spp)}{\deg(D_{j})}D_j|\) for all \(j\). Further, since \(d\) and \(\frac{\deg(\spp^0)}{\deg(D_{j})}\) are coprime to \(n\) (using that \(n\) divides \(\#J_C(k)\)), every point of degree \(\deg(\spp)\) is linearly equivalent to \(\frac{d\deg(\spp^0)}{\deg(D_j)}D_j\) for some \(j\). Hence, the splitting profile \(\spp\) is unlikely in \(C\).
\end{proof}

The careful reader will note that the proofs of Theorem~\ref{thm:IntroRank1Splitting} and~\ref{thm:IntroTotallyRamified} are very similar. Indeed, Theorem~\ref{thm:IntroTotallyRamified} could almost be a corollary of Theorem~\ref{thm:IntroRank1Splitting} by taking \(\spp^0= \{(\indx(C), 1)\} \) except that we have weakened the hypothesis that there are no points of degree at most \(\deg(\spp^0)=\indx(C)\) by instead assuming 
exactly what the proof needs: that \([\indx(C)] C \cap \Pic^{\indx(C)}(C) = \emptyset\).  This is the minimal hypothesis necessary for Theorem~\ref{thm:IntroTotallyRamified} (see Corollary~\ref{cor:Thm1.1Sharp}).

\subsubsection{Corollaries of Theorem~\ref{thm:unlikely_splitting}}\label{sec:CorOfUnlikelySplitting}
\begin{proof}[Proof of Theorem~\ref{thm:IntroUnlikelySplitting}]
    Since the Jacobian of \(C\) has rank \(0\), there are finitely many divisor classes of a given positive degree. Set \(w\colonequals \textup{index}(C)\) and let \(D_1, \dots, D_n\) be representatives for the complete list of \(k\)-rational divisor classes of degree \(w\).
     By assumption, \(\deg(D_j)|e^0_if_i\) for all \(j\) and \(i\). 
    Without loss of generality, assume that \(e^0_1f_1 = \min_i(e^0_if_i)\). Since we assumed that \(C\) has no points of degree at most \(\min(e^0_if_i) = e_1^0f_1\), we know that \(\frac{e^0_1f_1}{\deg(D_j)} D_j\) is ineffective for all \(j\).  Hence, for each \(j\), the locus 
    \[ (\iota_{D_j} \circ [\Diag(\be^0)])(W^{\bff}) \subset J_C^r\] 
    does not intersect \(\pr_1^{-1}(0)\), where \(\pr_1 \colon J_C^r \to J_C\) denotes projection onto the first factor. 

    Write \(s_0\) for the number of prime divisors of \(\#J_C(k)\).
    For each such prime divisor \(p_i|\#J_C(k)\), \(1 \leq i \leq s_0\), let \(M^0_i \subset \Z^r\) be the finite index subgroup of \(\Z^r\) consisting of tuples \((x_1, \dots, x_r)\) such that
    \[\sum_{\ell=1}^r \left(\frac{e_{\ell}^0f_{\ell}}{\indx(C)}\right) x_{\ell} \equiv 0 \pmod{p_i}.\]  
    If \(\be/\be^0 \in \Z^r \smallsetminus \cup_iM^0_i\), then \(\deg(\spp)/\indx(C) = \sum_{\ell=1}^r \left(\frac{e_{\ell}f_{\ell}}{\indx(C)}\right)\) is relatively prime to \(\#J_C(k)\) and so every point of degree \(\deg(\spp)\) is linearly equivalent to \(\frac{\deg(\spp)}{\deg(D_j)} D_j\) for some \(j\). 
    By our assumption that \(e_1^0f_1/ \indx(C)\) is relatively prime to \(\#J_C(k)\), we know that the \(M^0_i\) does not contain the first standard basis vector for any \(i\). 

    We claim that we can apply Theorem~\ref{thm:unlikely_splitting} for each \(D = D_j\) with \(H=0\) and \(\tilde{\be} = (1, 0, \dots, 0)\) the first standard basis vector.  Since \(J_C\) is geometrically simple by assumption, the only abelian subvarieties \(B \subset J_C\) that we need to consider are \(B = J_C\) and \(B = 0\).  For \(B = J_C\), we have \(\tau + B^r = B^r\) for every torsion point \(\tau\), and 
    \[
     \dim (\iota_D \circ [\Diag(\be^0)])(W^{\bff}) \cap B^r =\dim (\iota_D \circ [\Diag(\be^0)])(W^{\bff}) = \sum_{i=1}^r f_i < g = \dim B
    \]
    by assumption.  For \(B = 0\) and \(\tau \in J_C^r\) a torsion point, we have \(\dim ((\iota_D \circ [\Diag(\be^0)])(W^{\bff}) \cap (\tau + B^r)) \geq \dim B = 0\) if and only if \(\tau \in {(\iota_D \circ [\Diag(\be^0)])(W^{\bff})}\). Thus, we must show that whenever \(\tau \in (\iota_D \circ [\Diag(\be^0)])(W^{\bff})\), then \([\tilde{\be}]\tau \notin B = 0\). This follows since if \(\tau \in (\iota_D \circ [\Diag(\be^0)])(W^{\bff})\), then \([\tilde{\be}](\tau) = {\pr}_1(\tau)  \in {\pr}_1(\iota_D \circ [\Diag(\be^0)])(W^{\bff})\), and, as explained in the first paragraph, does not contain \(0\) by our assumptions. 
    Hence, for each \(j\), Theorem~\ref{thm:unlikely_splitting} with \(D = D_j\) and \(H = 0\) implies the existence of subgroups \(M^j_1, \dots, M^j_{s_j} \subset \Z^r\) such that \(\tilde{\be} \not\in M^j_i\) for each rank \(r\) subgroup \(M^j_i\) and for all
    \(\be \in \Diag(\be^0) (\Z^r \smallsetminus \bigcup M_i^j)\), the splitting profile \(\spp = \{(e_i, f_i)\}\) is unlikely in \(\left|\frac{\deg{\spp}}{\deg{D_j}} D_j\right|\).

    Consider the following finite unions of subgroups: 
    \[
    M \colonequals  \bigcup_{j={{0}}}^n \bigcup_{i=1}^{s_j} M^j_{i} \quad \textup{and} \quad \tilde{M} \colonequals \bigcup_{j={{0}}}^n \bigcup_{\substack{i=1\\\rank(M_i^j) = r}}^{s_j}M^j_{i}.
    \]
    By construction, \(\tilde{M}\not\ni \tilde{\be}\), and so \(\tilde{M}\) cannot equal all of \(\Z^r\). Since \(\Z^r\smallsetminus \tilde{M}\) must contain a coset of a finite index subgroup~\cite{VirayVogt}*{Lemma 4.3.8} and \(M\) is the union of \(\tilde{M}\) and finitely many subgroups of rank strictly less than \(r\), the finite union \(M\) also cannot equal all of \(\Z^r\) (even though it may contain \(\tilde{\be}\)).
    We claim that the theorem holds with this proper finite union of subgroups \(M\).
    Indeed, if \(\be \in \Diag(\be^0)(\Z^r \smallsetminus M)\), then every point of degree \(\deg(\spp)\) is in \(\left|\frac{\deg{\spp}}{\deg{D_j}} D_j\right|\) for some \(j\), so the result follows by Theorem~\ref{thm:unlikely_splitting}.
\end{proof}

\subsubsection{Proof of Table~\ref{table:genus2}}\label{sec:ProofOfTable}
In this Section we show how Theorem~\ref{thm:unlikely_splitting_rank1} yields the following explicit computational result:
\begin{prop}\label{prop:proof_of_table}
    Let \(C\) be a curve from Table~\ref{table:genus2}. Then for every odd \(d\), the set of places
    \[
\calT_d(C) := \left\{p : \exists x\in C, \; \deg(x) = d, \;\kk(x)/\Q\textup{ totally ramified at }p\right\}
\]
is finite. Moreover, for \(d = 3, 5\), the sets \(\calT_{d}(C)\cap \{p : p\nmid 2\cond(C)\}\) are as listed in column \(2\) and \(3\) of Table~\ref{table:genus2}.
\end{prop}
\begin{proof}
    From the L-functions and Modular Forms Database (LMFDB) and Poonen's \href{https://math.mit.edu/~poonen/papers/torsion.gp}{code} to compute torsion packets~\cites{LMFDB, Poonen-TorsionPackets, PARI}, each curve listed in Table~\ref{table:genus2} has the following properties (see below for more details):
\begin{enumerate}
    \item the Weierstrass points of \(C\) split into two degree \(3\) points, \(w, w'\) (which necessarily implies \([w] = [w']\) and \(2[w] = 3K_C\)), \label{eq:Weierstrass}
    \item the \(\Q\)-points of the Picard scheme are \(\{ nK_C, nK_C - [w] :n \in \Z\}\),\label{eq:PicardRank}
    \item for any geometric Weierstrass point \(p\in C(\kbar)\), the intersection \((C - p)\cap (J_C)_{\tors} \) consists only of the Weierstrass points.\label{eq:TorsionPacket}
\end{enumerate}
    Indeed, condition~\eqref{eq:Weierstrass} can be verified by taking a simple Weierstrass model \(y^2 = f(x)\) for \(C\) (which is given in the LMFDB) and observing that over \(\mathbb{Q}\), the polynomial \(f(x)\) factors as a product of two irreducible cubics. Then, given~\eqref{eq:Weierstrass}, which in particular implies that \(C\) has index \(1\), Condition~\eqref{eq:PicardRank} is equivalent to the condition that \(J_C(\mathbb{Q}) = \{O\}\), which is computed in the LMFDB. Finally,~\cite{Poonen-TorsionPackets} gives an algorithm to compute \((C - p)\cap (J_C)_{\tors}\), and we use Poonen's PARI/GP implementation to verify that this set of geometric points has order \(6\). Since this set contains every Weierstrass point (the difference of any two Weierstrass points is \(2\)-torsion), this proves that no other points are contained in this intersection.

    Now we will use these properties together with Theorem~\ref{thm:unlikely_splitting_rank1} to prove the proposition.  Condition~\eqref{eq:PicardRank} says that there is a unique divisor class in every degree, and they are all integer multiples of \(2K_C - [w]\). In addition, \(2K_C - [w] = [w] - K_C\) is not effective and so we may apply Theorem~\ref{thm:unlikely_splitting_rank1} with \(D\colonequals 2K_C - [w]\), \(\spp^0\colonequals\{(1,1)\}\), and \(H=0\). We claim that \(\iota_D(C)\cap J_C[d] = \emptyset\) for all odd \(d\), which implies that in Theorem~\ref{thm:unlikely_splitting_rank1} we may take \(m = 2\). Thus we obtain the desired conclusion about finiteness. We will prove this claim and then explain how we compute the sets \(\calT_{d}(C)\cap \{p : p\nmid 2\cond(C)\}\)
    
    Let \(q\in C(\kbar)\) be such that \([q] - [D]\in (J_C)_{{\tors}}\). Then \(2[q]-K_C\) is torsion, so for any geometric Weierstrass point \(p\in C(\kbar)\), the difference \([q]-[p]\) is torsion. By Condition~\eqref{eq:TorsionPacket}, \(q\) must be a geometric Weierstrass point, so \(2([q]-D) = K_C - K_C = 0\). Since \(D\) is ineffective, \([q]\ne D\) and so \(\iota_D(C)\cap J_C[d] = {\emptyset}\) for all odd \(d\).

    Now we explain how to compute the sets  \(\calT_{d}(C)\cap \{p : p\nmid 2\cond(C)\}\). Recall that condition~\eqref{eq:Weierstrass} implies that \(C\) has a model of the form \(y^2 = g_0(x)g_1(x)\) for irreducible cubics \(g_0,g_1\). Moreover, the polynomials \(g_0,g_1\) are separable and relatively prime for every odd prime not dividing the conductor of \(C\). The complete linear system \(|3D| = |w|\) corresponds to the morphism sending a point \(z\in C\) to \([g_0(z) : y(z)] = [y(z):g_1(z)]\). Thus if the fiber above a point \([t_0:t_1]\in \PP^1\) give a point that is totally ramified at \(v\) then the cubic polynomial \(t_1^2g_0 - t_0^2g_1\) is a scalar multiple of a perfect cube modulo \(v\). One can compute the locus of such \(v\) using elimination ideals (see the \texttt{Magma}~\cite{Magma} code in the \href{https://github.com/b-viray/unlikely-ramification/}{Github repository} for this paper). Further, for each of the \(v\) that appear, we verify that this value of \([t_0:t_1]\) can be lifted such that \(t_1^2g_0 - t_0^2g_1\) is irreducible over \(k_v\) and defines a totally ramified field extension.

    For  \(d=5\), we study the linear system \(|5D|\), which is very ample. Since \(h^0(5D) = h^0(2D)h^0(3D)\), the morphism corresponding to \(|5D|\) factors as the composition
    \[
        C\xrightarrow{(\varphi_{|2D|}, \varphi_{|3D|})} \PP^1\times\PP^1 \xrightarrow{(\varphi_{(1,1)})} \PP^3
    \]
    Thus, a hyperplane section corresponds to intersecting \(C\) (viewed as a \((2,3)\) curve) with a \((1,1)\) curve. We wish to apply Theorem~\ref{thm:Reconstruction}, so we must find a hyperplane section modulo \(v\) that meets \(C\) at an \(\F_v\)-point with multiplicity \(5\). Such a hyperplane section necessarily corresponds to an integral \((1,1)\) curve \(Z\simeq \PP^1\). Thus, we may compute the intersection of \(C\) with \(Z\) as the vanishing of a single variable degree \(5\) polynomial on \(Z\). By Lemma~\ref{lem:SplProfileImpliesFvPoints}, a place of good reduction \(v\in \calT_5(C)\) must give rise to a \(Z\) such that the associated degree \(5\) polynomial is a scalar multiple of a fifth power. As above, this can be computed using elimination ideals (see the \texttt{Magma}~\cite{Magma} code in the \href{https://github.com/b-viray/unlikely-ramification/}{Github repository} for this paper). For each of the \(v\) that appear, we could verify that the degree \(5\) polynomial can be lifted to one that is irreducible over \(k_v\) and defines a totally ramified field extension, as we did in the degree \(3\) case. Alternately, since \(|5D|\) is very ample, we can appeal to our lifting result (Theorem~\ref{thm:Reconstruction}) that we prove in the next section.
\end{proof}

Given a genus \(2\) curve \(C\) which has properties~\eqref{eq:Weierstrass} and~\eqref{eq:PicardRank}, the proof of Proposition~\ref{prop:proof_of_table} shows that if \(\iota_D(C)\) contains a (geometric) point of odd order \(d\geq 3\), then the Weierstrass torsion packet contains a (geometric) point of order \(2d\).  

\begin{example}
    Let \(C\) be the genus \(2\) curve with LMFDB label \href{https://www.lmfdb.org/Genus2Curve/Q/52969/b/52969/1}{52969.b.52969.1}. This curve has a geometric point whose difference with a geometric Weierstrass point is \({6}\)-torsion in the Jacobian, and whose difference with \(D\) is \(3\)-torsion. One can compute the set \(\calT_3(C)\) using elimination theory as described above, and in doing so, we determine that the set \(\calT_3(C)\cap \{p\nmid 2\cdot 52969\}\) consists of all primes different from \(2\) and \(23\) such that \(-7\in \F_p^{\times2}\), i.e., all primes different from \(2\) and \(23\) that are congruent to \(1, 2,\) or \(4\) modulo \(7\).
\end{example}

The methods of this paper do not directly shed light on the sets \(\calT_d(C)\) themselves (beyond finiteness). For example, to the best of our knowledge, the following questions remain open.
\begin{enumerate}
    \item Can the sets \(\calT_d(C)\) be arbitrarily large? If so, then for \(d\) in a fixed congruence class modulo \(m(C)\), what are the asymptotics of \(\#\calT_d(C)\)? 
    \item What is the Dirichlet density of \(\cup_{(d,m(C)) = 1} \calT_d(C)\)? We can computationally verify (see code in the \href{https://github.com/b-viray/unlikely-ramification/}{Github repository} for this paper) that for the curve \(C\) with LMFDB label \href{https://www.lmfdb.org/Genus2Curve/Q/7676/a/7676/1}{7676.a.7676.1}, the prime \(31\) is not in \(\calT_d(C)\) for any odd integer \(d\) (meaning there are no points on \(C\) of odd degree where \(31\) is totally ramified in the residue field). In particular, this example shows that \(\cup_{(d,m(C)) = 1} \calT_d(C)\) can be a proper subset of all primes of good reduction. However, beyond this, little is known.
\end{enumerate}

%%%%%%%%%%%%%%%%%%%%%%%%%%%%%%%%%%%%%%%%%%%%%%%%%%%%%%%%%%%%%%%%%%%%%%%%%%%%%%%%%%%%%%%%%%%%%%%%%%%%%%%%%%
\section{Realizing splitting profiles}\label{sec:Reconstruction}
%%%%%%%%%%%%%%%%%%%%%%%%%%%%%%%%%%%%%%%%%%%%%%%%%%%%%%%%%%%%%%%%%%%%%%%%%%%%%%%%%%%%%%%%%%%%%%%%%%%%%%%%%%

In this section, we explore the complementary problem of showing that certain splitting profiles can be realized. One of the main techniques in this section is a lifting result: that given an effective representative \(D\) of a sufficiently positive (globally defined) divisor class over a finite field \(\F\), the representative can be lifted to a point over a global field in the same linear system with local behavior prescribed by the divisor \(D\). Precisely, we prove the following.
\begin{theorem}\label{thm:Reconstruction}
    Let \(C/k\) be a nice curve, let \(v\) be a place of \(k\) such that \(C\) has good reduction at \(v\),  let \(D\) be a divisor on \(C\) of positive degree {\(d\)}, let \(\spp\) be a splitting profile with \(\deg(\spp) = {d}\) and let \(\mathbf{e} = \mathbf{e}(\spp), \mathbf{f}=\mathbf{f}(\spp)\) be the associated tuples.  Assume that \(r(\spp) < \#\F_v\), \(h^0(C, D) = h^0(C_{\F_{v}}, D_{\F_{v}})\) and \(|D_{\F_{v}}|\) is very ample. If there exist distinct points \(\overline{x}_1, \dots \overline{x}_{r(\spp)}\in C_{\F_v}\) with \((\deg(\overline{x}_i))_i = \mathbf{f}\) such that
            \begin{equation}\label{eq:sum_eipi}
                 \sum_i e_{i} \overline{x}_i \sim D_{{\F_v}}.
            \end{equation}
    then there exists a degree \(d\) point \(x\in C\) with splitting profile \(\spp\) at \(v\) and with \(x\sim D\).
\end{theorem}
This is morally a Hensel's Lemma and Hilbert irreducibility argument, although the proof requires care because we are lifting a (possibly) nonsmooth scheme. In Section~\ref{sec:LocalReconstruction}, we prove the nonarchimedean lifting result and in Section~\ref{sec:ReconstructionThm}, we explain how to deduce Theorem~\ref{thm:Reconstruction} from the local result. This lifting result allows us to give many examples of likely splitting profiles; these examples are collected in Section~\ref{sec:LikelyExamples}, along with a proof of Proposition~\ref{prop:IntroOnEAllProfilesLikely}.

%%%%%%%%%%%%%%%%%%%%%%%%%%%%%%%%%%%%%%%%%%%%%%%%%%%%%%%%%%%%%%%%%%%%%%%%%%%%%%%%%%%%%%%%%%%%%%%%%%%%%%%%%%
\subsection{Lifting effective divisors \(p\)-adically}\label{sec:LocalReconstruction}

Let \(K\) be a \(p\)-adic field with valuation ring \(\OO = \OO_K\) and residue field \(\F= \F_K\). Fix a choice of uniformizer \(\pi \in \OO\). Given a smooth \(\OO\)-scheme \(\calC \hookrightarrow\PP^N_{\calO}\) of relative dimension \(1\), the goal is to lift a linear section \(V(\overline{\ell})\subset \PP^N_{\F}\) that meets \(\calC_{\F}\) possibly with multiplicity, to a linear section over \(\OO\) that meets \(\calC\) with multiplicity \(1\) at every point in its support. Precisely, we prove the following

\begin{prop}\label{prop:LocalReconstruction}
    Let \(\calC \hookrightarrow\PP^N_{\OO}\) be a smooth \(\OO\)-scheme of relative dimension \(1\), let \(\spp\) be a splitting profile and let \(\mathbf{e} = \mathbf{e}(\spp), \mathbf{f}=\mathbf{f}(\spp)\) be the associated tuples.  Assume that \(r(\spp) < \#\F\)  and that there exists a linear form \(\bar{\ell} \in H^0(\PP^N_{\F}, \OO(1))\)  such that 
    \(\calC \cap V(\pi, \bar{\ell}) = e_1 \bar{x}_1 +\cdots+e_r \bar{x}_r\) where \(\bar{x}_1, \dots, \bar{x}_r \in \A^N_{\F} \subset \PP^N_{\F}\) are {distinct} points with \([\kk(\bar{x}_i):\F] = f_i\).
    Then there is an open neighborhood of linear forms \(\ell  \in H^0(\PP^N_{\OO}, \OO_{\PP^N}(1))\) such that \(\ell \equiv \bar{\ell} \pmod{\pi}\) and \(\calC_K \cap V(\ell) = x_1 + \dots + x_r\) where \(x_1, \dots, x_r \in \calC_K\) are distinct points and \(\kk(x_i)/K\) has ramification degree \(e_i\) and residual degree \(f_i\).
\end{prop}

The following criteria for when a quotient of a regular ring is also regular will be helpful in the proof.

\begin{lemma}\label{lem:RLR_condition}
    Let \((R,\mm)\) be a \(2\)-dimensional regular local ring. For any \(\ell\in \mm\smallsetminus \mm^2 \), the quotient \(R/\langle \ell\rangle\) is a \(1\)-dimensional regular local ring.
\end{lemma}
\begin{proof}
    Since \(\ell\in \mm\smallsetminus \mm^2\), there is some element \(z\in \mm\) such that \(\mm = \langle \ell, z\rangle\). Thus, the lemma follows from~\cite{Stacks}*{Lemma \href{https://stacks.math.columbia.edu/tag/00NQ}{10.106.3}}.
\end{proof}

\begin{proof}[Proof of Proposition~\ref{prop:LocalReconstruction}]
    Let \(B\) be the \(\OO\)-algebra such that \(\calC \cap \A^N_{\OO} = \Spec B\).  We may choose homogeneous coordinates \(z_0, \dots, z_N\) on \(\PP^N_{\OO}\) so that \(\bar{x}_1, \dots, \bar{x}_r \in D(z_0) \subset \PP^N_{\F}\) and \(\bar{\ell} = z_N\). This in particular allows us to treat \(z_1, \dots z_N\) as well-defined elements of \(B\).
    
    Let \(\frakm_1, \dots, \frakm_r \subset B\) be the maximal ideals corresponding to \(\bar{x}_1, \dots, \bar{x}_r\).  We have \((\pi, z_N) \subset \frakm_i\) for all \(i\).
    Since \(\calC \to \Spec \OO\) is smooth and \(\Spec \OO\) is regular, \(B_{\frakm_i}\) is a regular local ring of dimension \(2\)~\cite{Stacks}*{Lemma \href{https://stacks.math.columbia.edu/tag/07NF}{10.163.10}}.  
    In addition, since \(\calC_\F\) is a smooth \(\F\)-scheme, we have that \((B/(\pi))_{\bar{m}_i} \simeq B_{\frakm_i}/\pi B_{\frakm_i}\) is a regular local ring of dimension \(1\).  Hence, by \cite{Stacks}*{Lemma \href{https://stacks.math.columbia.edu/tag/00NR}{10.106.4}}, there exists an element \(g_i \in B_{\frakm_i}\) such that \((\pi, g_i)B_{\frakm_i} = \frakm_iB_{\frakm_i}\).  Since \(z_N \in \frakm_i\), we can write \(z_N = \alpha_i \pi + \beta_i g_i\) for some \(\alpha_i, \beta_i \in B_{\frakm_i}\).

    We will show that there exists a linear form \(\ell'=a_0 + a_1 z_1 + \cdots + a_{N-1} z_{N-1}\) (equivalently a choice of \(a_0, \dots, a_{N-1}\in \OO\)) such that for \(\ell = z_N + \pi \ell'\), we have that \((B/\ell)_{\bar{m}_i}\) is a regular local ring of dimension \(1\) (i.e., a discrete valuation ring~\cite{Stacks}*{Lemma \href{https://stacks.math.columbia.edu/tag/00PD}{10.119.7}}) for all \(1 \leq i \leq r\).  By Lemma~\ref{lem:RLR_condition}, this is equivalent to 
    \begin{equation}\label{eq:lin_form}
        \ell = z_N  + \pi(a_0 + a_1 z_1 + \cdots + a_{N-1} z_{N-1}) = \beta_i g_i + \pi(\alpha_i +a_0 + a_1 z_1 + \cdots + a_{N-1} z_{N-1})  \notin \frakm_i^2 B_{\frakm_i},
    \end{equation}
    for all \(i\). (Note that since \(\ell\) is a combination of \(z_N\) and \(\pi\), it is already in \(\frakm_iB_{\frakm_i}\) for each \(i\).)
    Since \(\pi, g_i\) form a basis for \((\frakm_iB_{\frakm_i}) / (\frakm_i^2B_{\frakm_i})\) as an \({\kk(\bar{x}_i)}\)-vector space, \(\ell\in \frakm_i^2B_{\frakm_i}\) if and only if \(\beta_i \in \frakm_iB_{\frakm_i}\) and 
        \begin{equation}\label{eq:hyperplane}
        (\alpha_i\bmod \frakm_i) +a_0 + a_1 (z_1\bmod \frakm_i) + \cdots + a_{N-1} (z_{N-1}\bmod \frakm_i) \in \frakm_i B_{\frakm_i}.
        \end{equation}
    Equation~\eqref{eq:hyperplane} only depends on the reduction of \(a_0, \dots, a_{N-1}\) in \(\OO/(\pi) = \F\), so the images of solutions \(a_0, \dots, a_{N-1}\) are contained in a translate of a hyperplane in \(\F^N\).
    Since \(\#\{i \leq r : \beta_i \in \frakm_i B_{\frakm_i}\} \leq r < \#\F\), the union of \(r\) translates of hyperplanes cannot cover \(\F^N\). Hence, there exists a choice of \(a_0, \dots, a_{N-1} \in \OO\) whose reduction modulo \(\pi\) does not satisfy \eqref{eq:hyperplane} for all \(i\) for which \(\beta_i \in \frakm_i B_{\frakm_i}\). Indeed, since the condition is on the reduction modulo \(\pi\), there is an open neighborhood \(U\subset \OO^N\) consisting of such \(a_0, \dots, a_{N-1}\). Setting 
    \(\ell  = \ell_{\mathbf{a}} = z_N + \pi(a_0 + a_1 z_1 + \cdots + a_{N-1}z_{N-1})\), we have shown that \(B_{\frakm_i}/(\ell B_{\frakm_i})\) is a discrete valuation ring for all \(\mathbf{a}\in U\).

    To simplify notation, let \(A \colonequals B/(\ell)\).
    Write \(\bar{\frakm}_i\) for the image of \(\frakm_i\) in \(A\).
    By \cite{Stacks}*{Proposition \href{https://stacks.math.columbia.edu/tag/00CT}{10.9.14}}, \(A_{\bar{\frakm}_i} \simeq B_{\frakm_i}/(\ell B_{\frakm_i})\), and so \(A_{\bar{\frakm}_i}\) is a discrete valuation ring with maximal ideal \(\bar{\frakm}_i\). In particular, \(A_{\bar{\frakm}_i}\) has a unique nonzero prime ideal (i.e., \(\bar{\frakm}_i\)), so we can refer to \emph{the} residue field of \(A_{\bar{\frakm}_i}\), and every nonzero ideal is a power of \(\bar{\frakm}_i\).
    
Since \(\ell \in \frakm_i\) we have  \(\kk(\bar{{x}}_i)=B/\frakm_i \simeq A/\bar{\frakm}_i\simeq A_{\bar{\frakm}_i}/\bar{\frakm}_iA_{\bar{\frakm}_i}\).  Thus the residue field of \(A_{\bar{\frakm}_i}\) is \(\kk(\bar{{x}}_i)\).  To prove that \(\pi A_{\bar{\frakm}_i} = \bar{\frakm}_i^{e_i} A_{\bar{\frakm}_i}\), it suffices to show that \(e_i\) is the smallest power of the image of \(\bar{\frakm}_i\) in \(A_{\bar{\frakm}_i}/(\pi A_{\bar{\frakm}_i})\) that is \(0\).
Repeatedly using the fact that localization respects quotients \cite{Stacks}*{Proposition \href{https://stacks.math.columbia.edu/tag/00CT}{10.9.14}} and also that \(\ell \equiv z_N \pmod{\pi}\), we have
\[A_{\bar{\frakm}_i}/(\pi A_{\bar{\frakm}_i}) \simeq (A/\pi A)_{\bar{\frakm}_i} \simeq B_{\frakm_i}/(\pi, \ell)B_{\frakm_i} = B_{\frakm_i}/(\pi, z_N)B_{\frakm_i}.\]
By assumption, \(\Spec(B_{\frakm_i}/(\pi, z_N)B_{\frakm_i})\) is nonreduced of order \(e_i\), and so \(\bar{\frakm}_i^{e_i} =0\) but \(\bar{\frakm}_i^{e_i - 1} \neq 0\).  All together, we have shown that the \(\OO\)-algebra \(A_{\bar{\frakm}_i}\) is a discrete valuation ring with residue field \(\kk(\bar{x}_i)\) and for which the valuation of \(\pi\) is exactly \(e_i\).

Finally, let \(\frakp_i \subset A\) be the (unique) height \(0\) prime ideal corresponding to \((0) \subset A_{\bar{\frakm}_i}\).  Then, since \(\frakp_i A_{\frakp_i} = 0\), 
\(A_{\frakp_i}\) is the fraction field of \(A_{\bar{\frakm}_i}\).
The generic fiber of the subscheme of \(\calC \cap V(\ell)\) corresponding to the prime \(\frakp_i\) is therefore a point \(x_i \in C_K \cap V(\ell)\) with 
\(\kk(x_i) =   A_{\frakp_i} = \Frac(A_{\bar{\frakm}_i})\).  We conclude that \(\kk(x_i)/K\) has ramification degree \(e_i\) and residual degree \(f_i\).
\end{proof}

%%%%%%%%%%%%%%%%%%%%%%%%%%%%%%%%%%%%%%%%%%%%%%%%%%%%%%%%%%%%%%%%%%%%%%%%%%%%%%%%%%%%%%%%%%%%%%%%%%%%%%%%%%
\subsection{Proof of Theorem~\ref{thm:Reconstruction}}\label{sec:ReconstructionThm}
%%%%%%%%%%%%%%%%%%%%%%%%%%%%%%%%%%%%%%%%%%%%%%%%%%%%%%%%%%%%%%%%%%%%%%%%%%%%%%%%%%%%%%%%%%%%%%%%%%%%%%%%%%

Under the assumptions, we have an embedding \(\calC \hookrightarrow \PP^{n}_{\OO}\) for \(n = h^0(C, D)-1\), and on the special fiber, the linear equivalence in~\eqref{eq:sum_eipi} guarantees that there exists a linear form in \(H^0(\PP^n_{\F}, \OO_{\PP^n_{\F}}(1))\) vanishing exactly on the divisor \(e_1 \overline{x}_1 + \cdots + e_r\overline{x}_r\). Since \(r < \#\F_v\), the points \(\overline{x}_1, \dots, \overline{x}_r\) are contained in some affine \(n\)-space within \(\PP^n\).
Thus, Proposition~\ref{prop:LocalReconstruction} yields an open set \(U\subset \HH^0\left(\PP^N_{\OO}, \OO_{\PP^N}(1)\right)\) of linear forms \(\ell\) such that \(V(\ell)\cap C\subset \PP^N_{k_v}\) consists of \(r\) closed points whose residue fields have ramification degree \(e_i\) and residual degree \(f_i\). Since \(k\) is dense in \( k_v\), there exists a \(k\)-rational linear form \(\ell\in U\). 
By projecting away from a suitable linear subspace of \(V(\ell)\), we obtain a \(k\)-rational map \(\pi\colon C \to \PP^1\) with \(C\cap V(\ell) = \pi^{-1}(\tilde{t})\) for some \(\tilde{t}\in \PP^1(k)\)~\cite{VirayVogt}*{Lemma 3.1.2}. 
Since \(\ell\in U\), \(\kk(\pi^{-1}{\tilde{t}})/k\) has splitting profile \(\spp\) at \(v\). Then by an extension of Hilbert irreducibility (see, e.g.,~\cite{BCLV}*{Lemma 5.2}), there exists a point \(t\in \PP^1(k)\) such that \(x:= \pi^{-1}(t)\) is integral (i.e., a point) and has the desired \(v\)-adic splitting behavior.\qed

\subsection{Examples of likely splitting profiles}\label{sec:LikelyExamples}

In this section, we show that likely splitting profiles in complete linear systems tend to propagate under taking multiples and repetitions.\footnote{If \(\spp^0\) is a splitting profile, then its multiple (by \(d\)), \(d\spp^0\) is \(\{(de,f) : (e,f)\in \spp^0\}\) and its (\(d\)-fold) repetition is \(\cup_{i=1}^d \spp^0\).}  For example, if an unramified splitting profile \(\spp^0\) is likely in a linear system \(|D|\), then sufficiently large multiples \(d\spp^0 = \{{(d,f) : f\in \bff(\spp^0)}\}\) are likely in \(|dD|\) (Corollary~\ref{cor:MultipleOfRealizedUnramIsRealized}). This propagation behavior allows us to show that the hypotheses of Theorem~\ref{thm:IntroTotallyRamified} are sharp in a strong sense and are key to the proof of Proposition~\ref{prop:IntroOnEAllProfilesLikely} (Section~\ref{sec:ProofOfIntroOnE}).

\subsubsection{Multiples of realized splitting profiles}\label{sec:MultiplesOfLikelyAreLikely}

\begin{lemma}\label{lem:ScalarMultOfRealizedIsRealized}
    Let \(C\) be a nice curve over a number field \(k\), let \(x\in C\), and let \(d\) be a positive integer.  Let \(v\) be a nonarchimedean place of \(k\) where \(C\) has good reduction  such that \(r(\spp_v(\kk(x)/k)) < \#\F_v\) and \(\Spec \OO_{\kk(x)}\to \calC\) is injective on geometric points over \(v\). If \(h^0(C, dx) = h^0(C_{\F_v}, (dx)_{\F_v})\) and \(|(dx)_{\F_v}|\) is very ample, then there exists a closed point \(y\in C\), linearly equivalent to \(dx\) such that
    \[
    \spp_v(\kk(y)/k) = d\spp_v(\kk(x)/k).
    \] 
\end{lemma}
\begin{proof}
    By Lemma~\ref{lem:SplProfileImpliesFvPoints}, there are distinct points \(\overline{x}_1, \dots, \overline{x}_{r}\in C_{\F_v}\) such that
    \[
        \deg(\overline{x}_i) = f_i \;\forall i, \quad\textup{and}\quad\sum_i e_i \overline{x}_i\sim x_{\F_v}.
    \]
    After multiplying both sides by \(d\), we can apply Theorem~\ref{thm:Reconstruction} to obtain a closed point \(y\in C\) such that \(\kk(y)\) has splitting profile \(d\spp_v(\kk(x)/k)\) at \(v\).    
\end{proof}

For unramified splitting profiles and totally ramified splitting profiles, we are often able to verify the hypotheses of Lemma~\ref{lem:ScalarMultOfRealizedIsRealized} for an infinite set of places (Corollaries~\ref{cor:MultipleOfRealizedUnramIsRealized} and~\ref{cor:MultipleOfRealizedTotalRamIsRealized}, respectively). From this lemma, we are also able to conclude that the hypotheses of Theorem~\ref{thm:IntroTotallyRamified} are sharp (Corollary~\ref{cor:Thm1.1Sharp}).

\begin{cor}\label{cor:MultipleOfRealizedUnramIsRealized}
    Let \(C\) be a nice curve over a number field \(k\), let \(x \in C\) be a closed point and let \(w \in \Omega_k\) be a place unramified in \(\kk(x)\).
    % If \(w \in \Omega_k\) is unramified in \(\kk(x)\), 
    Then, for all \(d \geq (2g(C) +1)/\deg(x)\), there exists a finite set of places \(S^0_d \subset \Omega_k\) such that
    \[
    \{v\notin S^0_d : \spp_v(\kk(x)/k) = \spp_w(\kk(x)/k)\} \subset \scrS_d \colonequals \{v : \exists y \in |dx| \textup{ integral}, \; \spp_v(\kk(y)/k) = d \spp_w(\kk(x)/k) \}.
    \]
    In particular, \(\scrS_d\) has positive density and the splitting profile \(d\spp_w(\kk(x)/k)\) is likely in \(|dx|\).
\end{cor}
\begin{proof}
    The last statement follows from the set containment since the Chebotarev density theorem says that the set \(\scrS \colonequals \{v : \spp_v(\kk(x)/k) = \spp_w(\kk(x)/k)\}\) has positive density (and, in particular, is infinite).  
    
    Let \(S\) be the (finite) set of places of bad reduction of \(C\), so that there is a smooth proper model \(\calC \to \Spec \OO_{k,S}\).
    Since \(\Spec \OO_{\kk(x), S}\to \calC\) is generically injective on geometric points, it is injective on geometric points above \(v\) for all places \(v\) away from a finite set \(S'\).  Let \(S''\) be the finite set of places \(v\) with \(\#\F_v \leq r(\spp_w(\kk(x)/k))\).  Thus, the hypotheses of  Lemma~\ref{lem:ScalarMultOfRealizedIsRealized} are satisfied for all \(v \in \scrS \smallsetminus (S \cup S' \cup S'')\) and \(d \geq (2g(C) +1)/\deg(x)\) by the Riemann--Roch theorem, and so Lemma~\ref{lem:ScalarMultOfRealizedIsRealized} guarantees the existence of a point \(y \in |dx|\) such that \(\spp_v(\kk(y)/k) = d\spp_v(\kk(x)/k) = d \spp_w(\kk(x)/k)\).  Hence, we have the desired containment with \(S^0_d := S\cup S'\cup S''\).
\end{proof}

\begin{cor}\label{cor:MultipleOfRealizedTotalRamIsRealized}
    Let \(C\) be a nice curve over a number field \(k\), let \(S\) be the finite set of bad primes of \(C\), let \(D\) be a divisor of positive degree, and let \(d \geq (2g(C) + 1)/\deg(D)\). Then
    \begin{equation}\label{eq:inclusion}
        \left\{v \not\in S : {\exists x \in |D| \textup{ integral}, \atop \spp_v(\kk(x)/k) = \{(\deg(D), 1)\}}\right\} \subset \left\{v : {\exists y \in |dD| \textup{ integral}, \atop \spp_v(\kk(y)/k) = \{(d\deg(D), 1)\}}\right\}
    \end{equation}
    In particular, if total ramification is likely in \(|D|\), then total ramification is likely in \(|dD|\).
\end{cor}
\begin{proof}
      Let \(v\) be a place in the left hand side set of \eqref{eq:inclusion}.  Then there exists a point \(x \in C\) such that \(\spp_v(\kk(x)/k) = \{(\deg(D), 1)\}\).  Since \(r(\{(\deg(D), 1)\}) = 1\) and \(\Spec \OO_{\kk(x)}\) only has one geometric point over \(v\), the hypotheses of Lemma~\ref{lem:ScalarMultOfRealizedIsRealized} are satisfied for all \(d \geq (2g(C) + 1)/\deg(D)\) by the Riemann--Roch theorem.  Hence \(v\) is in the right hand side set of \eqref{eq:inclusion}.
\end{proof}

\begin{cor}\label{cor:Thm1.1Sharp}
    Let \(C\) be a nice curve defined over a number field \(k\).
    Suppose that there is a geometric point \(x \in C(\bar{k})\) such that \(\indx(C)x\) is linearly equivalent to a \(k\)-rational divisor \(D\).  Then for all \(d \geq (2g(C) +1)/\indx(C)\), the set
    \[\scrS_d \colonequals \{v : \exists y \in |dD| \textup{ integral}, \; \spp_v(\kk(y)/k) = \{(d \indx(C), 1)\}\}\]
    has positive density.
    In particular, total ramification is likely in \(|dD|\).
\end{cor}
\begin{proof}
    Let \(L\) denote the field of definition of the geometric point \(x\) (i.e., the residue field of the corresponding closed point).
    Let \(v\) be a place of \(k\) that has a degree \(1\) place above it in \(L\) and at which \(C\) has good reduction.  Then there is a point \(\bar{x} \in C(\F_v)\) such that \(\indx(C) \bar{x} \in |D_{\F_v}|\).  Hence \(d \indx(C) \bar{x} \sim d D_{\F_v}\).  Further, \(h^0(C, dD) = h^0(C_{\F_v}, dD_{\F_v})\) and \(|{d}D_{\F_v}|\) is very ample by the Riemann--Roch theorem.  Therefore \(v \in \scrS_d\) by Theorem~\ref{thm:Reconstruction}.  Since the set of degree \(1\) places in \(L\) has positive density by the Chebotarev density theorem, we conclude that \(\scrS_d\) has positive density.
\end{proof}

\subsubsection{Unions of realized splitting profiles}

\begin{lemma}\label{lem:UnionOfRealizedIsRealized}
    Let \(D^{(1)}, D^{(2)}\) be divisors on \(C\) such that their sum is basepoint free and let \(\spp^{(1)}, \spp^{(2)}\) be splitting profiles with \(\deg(D^{(i)}) = \deg(\spp^{(i)})\). Fix a nonarchimedean place \(v\). If there exist points \(x^{(i)}\in |D^{(i)}| \subset C\) with \(\spp_v(\kk(x^{(i)})/k) = \spp^{(i)}\) and \(x^{(1)}\neq x^{(2)}\), then there exists a point \(y\in |D^{(1)}+ D^{(2)}|\) with splitting profile \(\spp^{(1)}\cup \spp^{(2)}\) at \(v\).
\end{lemma}
\begin{proof}
    Since \(x^{(1)}\in|D^{(1)}|\) and \(x^{(2)}\in |D^{(2)}|\), the sum \(x^{(1)}+x^{(2)}\) is an effective divisor in \(|D^{(1)} + D^{(2)}|\). Further, since \(|D^{(1)} + D^{(2)}|\) is basepoint free, there exists a morphism \(f\colon C \to \PP^1\) with \(\deg(f) = \deg(D^{(1)}) + \deg(D^{(2)})\) such that \(f^*0 = x^{(1)} + x^{(2)}\).  Since \(x^{(1)}\neq x^{(2)}\), \(f\) is unramified at \(0\) so in particular, the fiber is an \'etale algebra. By weak approximation combined with Hilbert irreducibility (see, e.g.,~\cite{BCLV}*{Lemma 5.2}), we may find a \(t\in \PP^1(k)\) such that \(y:=f^*t\) is integral and has \(\kk(y)\otimes k_v \simeq \kk(x^{(1)})\otimes k_v\times \kk(x^{(2)})\otimes k_v\). In particular, \(y\) has splitting profile \(\spp^{(1)} \cup \spp^{(2)}\) at \(v\).
\end{proof}

Using this, we will be able to show that sufficiently many repetitions of a likely splitting profile is likely (Proposition~\ref{prop:DuplicationsofLikelyAreLikely}).  We begin with the following lemmas.

\begin{lemma}\label{lem:fibersvary}
    Let \(C\) be a nice curve over a number field \(k\) and let \(C \to \PP^1\) be a degree \(d\) morphism.  For any finite collection of degree \(d\) extensions \(k_1, \dots, k_n\) over \(k\), there exist infinitely many places \(v \in \Omega_k\) that split completely in \(k_i\) for all \(1 \leq i \leq n\) and points \(t = t(v) \in \PP^1(k)\) such that \(C_t\) is integral and \(\kk(C_t)_v\) is nonsplit. In particular, \(\kk(C_t) \not\simeq k_i\) for any \(1\leq i \leq n\).
\end{lemma}
\begin{proof}
    For all sufficiently large \(v\), there exists a \(k_v\) point on a nontrivial  \'etale twist of the Galois closure \(C\to \PP^1\). In other words, for all sufficiently large \(v\), there exists a point \(t = t(v) \in \PP^1(k)\) such that \(\kk(C_t)_v\) is a nontrivial \'etale \(k_v\)-algebra. If we additionally constrain \(v\) so that it splits completely in \(k_i\) for all \(i\), then \(\kk(C_t)\not\simeq k_i\) as desired.
\end{proof}

\begin{lemma}\label{lem:SplittingProfileRealizedByManyPoints}
    Let \(C\) be a nice curve over a number field \(k\), let \(x\in C\) be a closed point, and let \(v\) be a nonarchimedean place of \(k\). If \(|x|\) is basepointfree, then there exist infinitely many closed points \(y\in C\) such that \(y\sim x\) and \(\spp_v(\kk(x)/k) = \spp_v(\kk(y)/k)\).
\end{lemma}
\begin{proof}
    If \(|x|\) is basepointfree then there exists an effective divisor \(D\in |x|\) whose support is disjoint from \(x\). Thus there is a function \(f\colon C \to \PP^1\) of degree \(d\) such that \(f^*0 = x\) and \(f^*\infty = D\). Now choose a place \(w{\neq v} \in \Omega_k\) and point \(t_0\in \PP^1(k)\) such that \(t_0\) is not a branch point of \(f\) and \(\kk(f^*t_0)\otimes k_w\not\simeq \kk(x)\otimes k_w\), which exist by Lemma~\ref{lem:fibersvary}. Then by an extension of Hilbert irreducibility~\cite{BCLV}*{Lemma 5.2} there exists a closed point \(t\in \PP^1(k)\) such that the fiber \(y:= f^*t\) is integral and 
    \[
    \kk(y)\otimes k_v\simeq \kk(x)\otimes k_v \quad \kk(y)\otimes k_{{w}}\not\simeq \kk(x)\otimes k_w.
    \]
    In particular, \(y\sim x\), \(y\neq x\), and \(\spp_v(\kk(x)/k) = \spp_v(\kk(y)/k)\).

    To obtain infinitely many such points, we iteratively add more conditions at more places to ensure that each new point \(y_i\) is different than those previously obtained.
\end{proof}

\begin{prop}\label{prop:DuplicationsofLikelyAreLikely}
    Let \(C\) be a nice curve over a number field \(k\) and let \(D\) be a divisor on \(C\).
    Suppose that \(\spp^{0}\) is likely in \(|D|\) and that \(|D|\) is basepoint free.  Then for all \(d \geq 1\), \(\bigcup_{i=1}^d\spp^{0}\) is likely in \(|dD|\).
\end{prop}
\begin{proof}
    Let \(v\) be one of the (infinitely many) places for which there exists a point \(x\in |D|\) with \(\spp_v(\kk(x)/k) = \spp^0\).  By assumption \(|x|\) is basepoint free.  Hence, by Lemma~\ref{lem:SplittingProfileRealizedByManyPoints}, there exist distinct points \(y_1 \neq \cdots \neq y_d\) such that \(y_i \in |D|\) and \(\spp_v(\kk(y_i)/k) = \spp_v(\kk(x)/k) = \spp^0\).  We now iteratively apply Lemma~\ref{lem:UnionOfRealizedIsRealized}.
\end{proof}

\begin{remark}
    The hypothesis that \(|D|\) is basepoint free in Proposition~\ref{prop:DuplicationsofLikelyAreLikely} is not very restrictive: If \(\spp^0\) is likely in \(|D|\) but \(|D|\) is not basepoint free, then, since \(|D|\) necessarily contains a closed point \(x\), we must have \(|D| = \{x\}\).  Hence \(\spp^0\) is unramified, and so is \(\bigcup_{i=1}^d \spp^0\) for all \(d \geq 1\).  Once \(d \geq (2 g(C) + 1)/\deg(D)\), the fact that \(\bigcup_{i=1}^d \spp^0\) is likely follows from Proposition~\ref{prop:UnramifiedOfLargeDegreeLikely} below. 
\end{remark}

\subsubsection{Sufficiently large unramified splitting profiles are likely}\label{sec:UnramifiedOfLargeDegreeLikely}

\begin{lemma}\label{lem:Sd_unram}
    Let \(C\) be a nice curve over a number field \(k\) and let \(\psi \colon C \to \PP^1\) be a degree \(d\) cover with (geometric) Galois group \(S_d\).  Let \(\spp\) be an unramified splitting profile.  For all sufficiently large \(w \in \Omega_k\), there exists a point \(t \in \PP^1(k)\) such that the fiber \(y \colonequals C_{t}\) is integral, \(\spp_w(\kk(y)/k) = \spp\), and the map \(\Spec \OO_{\kk(y)} \to \calC\) is injective on geometric points over \(w\).
\end{lemma}
\begin{proof}
    Write \(\bff(\spp) = (f_1, \dots, f_r)\).
    For \(w\) sufficiently large, the map \(\psi\) has good reduction at \(w\) and the special fiber \(\psi_{\F_w} \colon C_{\F_w} \to \PP^1_{\F_w}\) is an \(S_d\)-cover \cite{BCLV}*{Lemma 5.1}.  By the Chebotarev density theorem, there exists a point \(\bar{t}_w \in \PP^1(\F_w)\) such that the fiber \(C_{\bar{t}_w}\) consists of points \(\bar{x}_1, \dots, \bar{x}_{r}\) with \(\deg(\bar{x}_i) = f_i\).  In particular, \(e(\bar{x}_i/\bar{t}_w) = 1\).  Let \(t_w \in \PP^1(k_w)\) specialize to \(\bar{t}_w\). By \cite{BCLV}*{Theorem 1.1}, the field \(\kk((C_{k_w})_{t_w})\) is a product of \(r\) unramified extensions of \(k_w\) of degrees \(f_1, \dots, f_r\).  By Hilbert's irreducibility theorem combined with weak approximation (see \cite{BCLV}*{Lemma 5.2}), there exists a point \(t \in \PP^1(k)\) reducing to \(\bar{t}_w\) such that \(y \colonequals C_t\) is integral and \(\kk(y)\otimes_k k_w \simeq \kk((C_{k_w})_{t_w})\).  In particular, \(\spp_w(\kk(y)/k) = \spp\), and since the fiber \(C_{\bar{t}_w}\) is reduced, the map \(\Spec \OO_{\kk(y)} \to \calC\) is injective on geometric points over \(w\).
\end{proof}

\begin{prop}\label{prop:UnramifiedOfLargeDegreeLikely}
    Let \(C\) be a nice curve over a number field \(k\), let \(D\) be a very ample divisor. Then for every  unramified splitting profile \(\spp\) with \(\deg(\spp) = \deg(D)\) the set
    \[
    \{v\in \Omega_k :  \exists y\in {C}, y\sim D, \Spec\OO_{\kk(y)} \to \calC \text{ inj.~on geom.~pts at \(v\), } \spp_v(\kk(y)/k) = \spp\}
    \]
    contains all sufficiently large primes.
\end{prop}
\begin{proof}
    Since \(|D|\) is assumed to be very ample, the incidence correspondence \(I\subset C\times |D|\) is a {geometric} \(S_d\)-cover of \(|D|\)~\cite{ACGH}*{Lemma, Page 111}. Let \(U \subset |D|\) be such that \(I_U \to U\) is \'etale.  Let \(\tilde{I}_U\) denote the \(d\)-fold fiber product of \(I_U\) with itself over \(U\), minus the big diagonal.  Since \(I \to |D|\) is a geometric \(S_d\)-cover, \(\tilde{I}_U\) is geometrically irreducible.  By Bertini's irreducibility theorem applied to \(\tilde{I}_U \to U \hookrightarrow |D|\), there is a dense open set of lines \(L \subset |D|\) such that \(\tilde{I}_U|_{U \cap L} \to U \cap L\) remains irreducible.  Hence, we can find a \(k\)-rational line in \(L\subset |D|\) such that \(I_L\to L\) remains an \(S_d\)-cover. 
    Away from a finite set \(S\) of places, there is an integral inclusion \(\calI_{L} \hookrightarrow \calC \times \PP^1_{\OO_{k,S}}\). 
    Applying Lemma~\ref{lem:Sd_unram} to the map \(I_L \to L\simeq \PP^1\) we conclude that, for every sufficiently large prime \(w\),
    there exists a \(t\in L(k)\), such that \(I_t\) is integral and has splitting profile \(\spp\) at \(w\) and \(\Spec \OO_{\kk(I_t), S} \to \calI_L\) is injective on geometric points over \(w\).
    Then the projection of \(I_t\) onto \(C\) is a point \(y\in C\) with \(\deg(y) = \deg(I_t)\), hence \(y\simeq I_t\) as \(k\)-schemes. In particular, \(y\) also has splitting profile \(\spp\) at \(w\) and the composition \(\Spec \OO_{\kk(y)} \to \calI_t \hookrightarrow \calC \times t\) is injective on geometric points over \(w\).
\end{proof}

\subsubsection{Proof of Proposition~\ref{prop:IntroOnEAllProfilesLikely}}\label{sec:ProofOfIntroOnE}
The statement is trivial when \(\deg(\spp) = 1\). If \(\deg(\spp) = 2\), then \(\phi_{|2x|}\colon E \to \PP^1\) is a double cover ramified at \(\infty\). Thus, for all but finitely many places \(v\), there exists \(t\in \PP^{1}(k)\), \(v\)-adically close to \(\infty\) such that \(E_t\) is integral and \(\kk(E_t)\) is (totally) ramified at \(v\). 
By Lemma~\ref{lem:Sd_unram}, for all but finitely many places \(v\), there exists \(t_1, t_2\in \PP^{1}(k)\) such that \(E_{t_i}\) is integral, \(\kk(E_{t_1})\) is split at \(v\), and \(\kk(E_{t_2})\) is inert at \(v\) and the corresponding maps \(\Spec \OO_{\kk(E_{t_i})} \to \calE\) are injective on geometric points over \(v\). 
Thus Proposition~\ref{prop:IntroOnEAllProfilesLikely} holds for all degree \(2\) splitting profiles, and, Lemma~\ref{lem:ScalarMultOfRealizedIsRealized} extends this to all multiples of degree \(2\) splitting profiles.

Assume that all ramification degrees in \(\spp\) are equal, i.e., that \(\be = (d, d, \dots, d)\) for some fixed positive integer. Thus, we may think of \(\spp\) as a multiple of an unramified profile \(\spp'\), i.e., 
\[
    \bff(\spp) = \bff(\spp'), \quad, \be(\spp') = (1,\dots, 1), \quad \be(\spp) = d\be(\spp')
\]
If \(\deg(\spp')\geq 3\), then \(\deg(\spp')x\) is very ample and so Proposition~\ref{prop:IntroOnEAllProfilesLikely} follows from Proposition~\ref{prop:UnramifiedOfLargeDegreeLikely} and Lemma~\ref{lem:ScalarMultOfRealizedIsRealized}. If \(\deg(\spp') = 2\), then the argument in the previous paragraph applies. If \(\deg(\spp') = 1\) (i.e., \(\spp' = \{(1,1)\}\)), then we apply Corollary~\ref{cor:MultipleOfRealizedTotalRamIsRealized} if \(d\geq 3\) and use the argument in the first paragraph if \(d\leq {2}\).

If there are at least two distinct ramification degrees in \(\spp\), then we may write the profile as a disjoint union of profiles \(\spp^{(1)}, \dots, \spp^{(n)}\) each with a fixed ramification degrees. By the argument in the previous paragraph, for every sufficiently large prime \(v\), we have a point \(y_i\in E\) such that \(\kk(y_i)/k\) has splitting profile \(\spp^{(i)}\) at \(v\). Since \(\spp^{(i)}\neq\spp^{(j)}\) for \(i\neq j\), \(y_i\neq y_j\). Thus we may apply Lemma~\ref{lem:UnionOfRealizedIsRealized} iteratively to obtain a point \(y\in E\) with splitting profile \(\spp = \spp^{(1)}\cup \dots \cup \spp^{(n)}\) at \(v\).  \qed

%%%%%%%%%%%%%%%%%%%%%%%%%%%%%%%%%%%%%%%%%%%%%%%%%%%%%%%%%%%%%%%%%%%%%%%%%%%%%%%%%%%%%%%%%%%%%%%%%%%%%%%%%%
\section{Sets of places realizing a given splitting profile}\label{sec:LikelyPositiveDensity} 
The goal of this section is to prove Corollary~\ref{cor:MaximalLikelyIsPositiveDensity} (which implies Theorem~\ref{thm:IntroLikelyFrobenian}), i.e., to understand the density of the set of places realizing a maximal likely splitting profile. To do so, we need to understand how the converse of Lemma~\ref{lem:SymIntersectionEmptyImpliesUnlikely} can fail.

    \begin{defn}
        Let \(H\subset \Pic^0_C\) be an abelian subvariety and \(D\) a divisor of positive degree. Fix a splitting profile \(\spp\) with \(\deg(D)=\deg(\spp)\).
        Define \(Y_{\spp, D, H} \subset \Sym^{\bff(\spp)}_C\) to be the preimage of \(H + [D]\) under the composition 
        \[
        \Sym^{\bff(\spp)}_C\xrightarrow{\textup{Abel--Jacobi}}\Pic^{\bff(\spp)}\xrightarrow{[\be(\spp)]}\Pic^{\deg(\spp)}_C
        \]
    \end{defn}    
    Lemma~\ref{lem:SymIntersectionEmptyImpliesUnlikely} states that if a splitting profile \(\spp\) is likely in \(|D + H|\), then \(Y_{\spp, D, H}\neq \emptyset\). However, the converse does fail. For example, if \(x,y\in C(k)\) are such that \(|x+y| = \{x+y\}\) and \(\spp = \{(1,2)\}\), then \(Y_{\spp, [x] + [y], 0}\neq\emptyset\), but since \(|x + y|\) consists of the single reducible divisor \(x + y\), the splitting profile \(\spp\) is not realized by a single place. The issue is that the point of \(Y_{\spp, [x] + [y], 0}\) really arises from the profile \(\spp' := \{(1,1),(1,1)\}\). To explore this further, we fix some notation.

\begin{remark}
    The proof of Theorem~\ref{thm:IntroLikelyFrobenian} will be in the case when $H=0$.
\end{remark}

\begin{definition}\label{defn:Covering}
    The splitting profile \(\spp'\) \defi{covers} the splitting profile \(\spp\), denoted \(\spp' \gtrdot \spp\), if for some \(e, e', f, f' \in \Z_{>0}\) we have one of the following conditions
    \begin{enumerate}
    \item \label{it:Coveringfs}
    \(\spp' \smallsetminus (\spp\cap \spp') = \{(e, f), (e, f')\}, \quad \text{and} \qquad \spp \smallsetminus (\spp\cap \spp') = \{(e, f+f')\},\) or
    \item \(\spp' \smallsetminus (\spp\cap \spp') = \{(e + e', f)\},\phantom{i} \qquad \text{and} \qquad \spp \smallsetminus (\spp\cap \spp') = \{(e, f), (e', f)\}.\) \label{it:Coveringes}
    \end{enumerate}
    The covering relation \(\gtrdot\) generates a partial order \(\geq\) on the set of splitting profiles: we say that \(\spp'> \spp\) if there is a sequence of covering relations, \(\spp':=\spp^{(n)}\gtrdot \cdots \gtrdot \spp^{(1)}=\spp\). If \(\spp' \geq \spp\), then \(\deg(\spp') = \deg(\spp)\), \(\sum_{e_i\in \mathbf{e}(\spp')} e_i \geq \sum_{e_i\in \mathbf{e}(\spp)}e_i\) and \(\sum_{f_i\in \mathbf{f}(\spp')} f_i \leq \sum_{f_i\in \mathbf{f}(\spp)}f_i\). Additionally, any splitting profile \(\spp\) sits between a totally ramified splitting profile and a profile with maximum residual degree:
    \[
    \{(\deg(\spp),1)\} \geq \spp \geq \{(1, \deg(\spp))\}.
    \]
\end{definition}

\begin{lemma}\label{lem:MapOfYPD}
    Let \(\spp'\gtrdot\spp\) be splitting profiles. Then there is a commutative diagram
    \[
    \begin{tikzcd}
    \Sym^{\bff(\spp')}_C\arrow[rr, "\textup{Abel--Jacobi}"] \arrow[d]& & \Pic_C^{\bff(\spp')}\arrow[rd, "{[\be(\spp')]}"] \arrow[d]\\
    \Sym^{\bff(\spp)}_C\arrow[rr, "\textup{Abel--Jacobi}"] & & \Pic_C^{\bff(\spp)}\arrow[r, "{[\be(\spp)]}"] & \Pic^{\deg(\spp)}_C
    \end{tikzcd}
    \]
    If~\eqref{it:Coveringfs} holds, then the vertical maps are given by partial summation and the leftmost map is finite surjective. If~\eqref{it:Coveringes} holds, then the vertical maps are given by a partial diagonal and are closed immersions. In particular, for every abelian subvariety \(H\subset \Pic^0_C\) and every divisor \(D\) of positive degree with \(\deg(D)= \deg(\spp)\), we have a  finite morphism
    \[
        Y_{\spp', D, H} \to Y_{\spp, D, H}.
    \]
\end{lemma}
\begin{proof}
    It suffices to prove this in the case that $\spp\cap \spp' = \emptyset$. In that case, the first three statements of the lemma follow from the two  diagrams (\eqref{it:Coveringfs} on the left and~\eqref{it:Coveringes} on the right):
        \[
    \begin{tikzcd}
    \Sym^{(f,f')}_C\arrow[r] \arrow[d, "+"]&  \Pic_C^{(f,f')}\arrow[rd, "{[(e,e)]}"] \arrow[d, "+"]\\
    \Sym^{f + f'}_C\arrow[r] & \Pic^{f+f'}_C\arrow[r, "{[e]}"] & \Pic_C^{\deg(\spp)}
    \end{tikzcd}
    \quad
        \begin{tikzcd}
    \Sym^{f}_C\arrow[r] \arrow[d, "\Delta"]&  \Pic^{f}_C\arrow[rd, "{[e+e']}"] \arrow[d, "\Delta"]\\
    \Sym^{(f,f)}_C\arrow[r] & \Pic^{(f,f)}_C\arrow[r, "{[(e,{e'})]}"] & \Pic^{\deg(\spp)}_C
    \end{tikzcd}
    \]
    which evidently have the desired properties. Since \(Y_{\spp', D,H}\) and \(Y_{\spp, D, H}\) are both complete preimages of \([D]+H\), the last claim follows from the previous ones.
\end{proof}

With this framework, we can determine the extent to which a converse of Lemma~\ref{lem:SymIntersectionEmptyImpliesUnlikely} holds.
\begin{prop}\label{prop:SymIntersectionNonemptyImpliesARefinementIsLikely}
    Let \(D\) be a very ample divisor and let \(\spp\) be a splitting profile with \(\deg(D) = \deg(\spp)\). If \(Y_{\spp, D, \{0\}}\neq\emptyset\) then there exists a splitting profile  \(\spp'\geq \spp\) such that \(\spp'\) is likely in \(|D|\).
\end{prop}
The converse of this proposition holds by Lemmas~\ref{lem:SymIntersectionEmptyImpliesUnlikely} and~\ref{lem:MapOfYPD}. Before proving the proposition, we first record the following corollary, which easily implies Theorem~\ref{thm:IntroLikelyFrobenian}.
\begin{cor}\label{cor:MaximalLikelyIsPositiveDensity}
    Let \(D\) be a very ample divisor on a nice curve \(C/k\) and let \(\spp\) be a splitting profile with \(\deg(D)=\deg(\spp)\). Assume that \(\spp\) is maximal among splitting profiles that are likely in \(|D|\), i.e., assume that \(\spp\) is likely in \(|D|\), and that for all \(\spp'>\spp\), the profile \(\spp'\) is unlikely in \(|D|\). Then the set
    \begin{equation}
    \calT_{\spp, D}:= \left\{v \in \Omega_k  :  \exists x\in C, \; x\sim D
            \textup{ with }\spp_v(\kk(x)/k) = \spp\right\}
    \end{equation}
    is infinite and almost equal to the set \(\{v: Y_{\spp, D, \{0\}}(\F_v)\neq\emptyset\}\). Furthermore, \(\calT_{\spp, D}\) is Frobenian with positive density; it has density \(1\) if \(Y_{\spp, D, \{0\}}\) is \defi{pseudo-split}.
\end{cor}
\begin{remark}
    In order to extend Corollary~\ref{cor:MaximalLikelyIsPositiveDensity} to other splitting profiles, one needs to have a deep understanding of when the existence of an \(\F_v\) point on \(Y_{\spp, D, \{0\}}\) is the limit of a closed point \(x\in C\) with \(\spp > \spp_v(\kk(x)/k)\). For example, let \(C\subset \PP^2\) be the smooth plane quartic defined by \((z_0^2 + z_1^2)^2 + z_0z_2(z_1^2 + z_2^2) = 0\). The line \(z_0=0\) meets \(C\) at the single point \([0:0:1]\) with multiplicity \(4\), which shows that \(Y_{\{(4,1)\}, K_C, \{0\}}(\F_v)\neq\emptyset\) for all \(v\). Additionally, the line \(z_2=0\) meets \(C\) at the degree \(2\) point \(V(z_0^2 + z_1^2, z_2)\) with multiplicity 2. Thus, for any \(v\) that splits in \(\Q(i)\), we have \(Y_{\{(2,1),(2,1)\}, K_C, \{0\}}(\F_v)\neq\emptyset\). Lemma~\ref{lem:SplProfileImpliesFvPoints} shows that these \(\F_v\)-points of \(Y_{\{(2,1),(2,1)\}, K_C, \{0\}}\) (the ones corresponding to the bitangent \(z_2=0\)) will never arise as the limit of a closed point \(x\in |K_C|\) with splitting profile \(\{(2,2)\}\). However, the \(\F_v\)-points of \(Y_{\{(4,1)\}, K_C, \{0\}}\) \emph{can} arise as such a limit. Indeed, that happens for \(v = 7\) and \(x = C\cap V(z_0 - 49z_2)\). 
\end{remark}
\begin{proof}[Proof of Theorem~\ref{thm:IntroLikelyFrobenian}]
    Since \(J_C(k)\) has rank \(0\), there are finitely many \(k\)-divisor classes \([D_1], \dots, [D_n]\) of degree \(d\), which are guaranteed to be very ample by the Riemann--Roch theorem.
    If \(\spp \colonequals \{(d,1)\}\) is likely in \(|D_i|\),
    we apply Corollary~\ref{cor:MaximalLikelyIsPositiveDensity} 
    with \(D= D_i\), 
    which is valid since there are no splitting profiles \(\spp' > \spp\). 
    If \(\spp\) is unlikely in \(|D_i|\), then \(\calT_{\spp, D_i}\) is finite, hence Frobenian.  All together,
    this implies the result since the finite union of Frobenian sets is Frobenian.
\end{proof}
\begin{proof}[Proof of Corollary~\ref{cor:MaximalLikelyIsPositiveDensity}]
The second statement follows from the first by work of Loughran--Smeets and Loughran--Skorobogatov--Smeets, specifically~\cite{LS-Fibrations}*{Lemma 3.9} and \cite{LSS-ArithmeticSurjectivity}*{Paragraph before Proposition 2.11}.\footnote{If \(Y_{P,D,\{0\}}\) is not smooth, we may write it as a finite union of smooth varieties, apply the results from~\cites{LS-Fibrations, LSS-ArithmeticSurjectivity} to each, and use that a finite union of Frobenian sets is Frobenian.} It remains to prove the first statement. Since we work with a fixed divisor \(D\) and \(H = \{0\}\), for any splitting profile \(\spp'\) with \(\deg(\spp') = \deg(D)\), we set \(Y_{\spp'} := Y_{\spp', D, \{0\}}\). 

Fix \(v\in \calT_{\spp, D}\). With only finitely many exceptions, we may apply Lemma~\ref{lem:SplProfileImpliesFvPoints} and conclude that \(Y_{\spp}(\F_v)\neq\emptyset\). 
Now let \(v\) be a nonarchimedean place such that \(Y_{\spp}(\F_v)\neq\emptyset\), in other words we have effective divisors \(z_1, \dots, z_r\) (where \(r\colonequals r(\spp)\)) such that \(\deg(z_i) = f_i\) and \(\sum_i e_i [z_i] = [D_{\F_v}]\). Let \(\spp'\) be a splitting profile with \(\spp'>\spp\). Then, by assumption,  \(\spp''\) is unlikely for all \(\spp''\geq \spp'\). Thus, by Proposition~\ref{prop:SymIntersectionNonemptyImpliesARefinementIsLikely}, \(Y_{\spp'}=\emptyset\). In particular,
\[
\bigcup_{\spp'>\spp}\left\{v\in \Omega_k : Y_{\spp'}(\F_v)\neq\emptyset\right\}
\]
is finite. Hence, away from finitely many places, the collection \((z_i)\in Y_{\spp}\) is not in the image of an \(\F_v\)-point on \(Y_{\spp'}\) for any \(\spp'>\spp\). In other words, each \(z_i\) is integral and (after possibly excluding finitely many more places) we may apply Theorem~\ref{thm:Reconstruction} to conclude that \(v\in \calT_{\spp,D}\).
\end{proof}

\begin{proof}[Proof of Proposition~\ref{prop:SymIntersectionNonemptyImpliesARefinementIsLikely}]
    If \(Y_{\spp, D, \{0\}}\neq \emptyset\), then \(Y_{\spp, D, \{0\}}\) has closed points. Fix some closed point \(y\) and set \(L\colonequals \kk(y)\), which is a finite extension of \({k}\). Thus, over \(L\), \(Y_{\spp, D, \{0\}}\) has a rational point which, by definition, corresponds to effective divisors \(Z_1, \dots, Z_{r(\spp)}\) (defined over \(L\)) where \(\deg(Z_i) = f_i\) and \(\sum_i e_iZ_i\simeq D\). We claim that after possibly replacing \(\spp\) with a splitting profile greater than it, we may assume that each \(Z_i\) is integral and \(Z_i\neq Z_j\) for all \(i\neq j\). Indeed, if \(Z_i = \sum_j n_{i,j}Z'_{i,j}\), then the effective divisors \((Z'_{i,j})_{i,j}\) give an \(L\)-point on \(Y_{\spp', D, \{0\}}\) where
    \[
    \spp' := \bigcup_{i = 1}^{r(\spp)} \bigcup_{j} \left\{(e_in_{i,j}, \deg(Z'_{i,j}))\right\}.
    \]
    By definition of the ordering, \(\spp'\geq \spp\). Further, if some \(Z'_{i,j}=Z'_{i',j'}\) for \((i,j)\neq (i',j')\), then we may replace \(\spp'\) with one that covers it using the relation in Definition~\ref{defn:Covering}\eqref{it:Coveringes}. 

    Henceforth, we assume that each \(Z_i\) is integral and set \(\spp'\colonequals \cup_{(e,f)\in \spp} \left\{(e,1)\right\}_{j=1}^f\). Now let \(v\) be any place of \(k\) that is split completely in \(L\) and split completely in \(\kk(Z_i)/L\) for all \(i\); by the Chebotarev Density Theorem there are a positive density set of such places \(v\). Further exclude the finitely many places where either \(C\) has bad reduction, where \(Z_i \bmod v\) is nonreduced for some \(i\), or where \(Z_i\bmod v, Z_j\bmod v\) have overlapping support for some \(i\neq j\). Then for any \(v\) in this positive density of places, the divisors \(Z_i\) give a \(k_v\)-point on \(Y_{\spp', D, \{0\}}\) that is not in the image of \(Y_{\spp'', D, \{0\}}\) for any \(\spp''> {\spp'}\). Furthermore, the reduction of these divisors \(Z_i\) give \(\F_v\)-points satisfying the equation~\eqref{eq:sum_eipi} from the hypothesis of Theorem~\ref{thm:Reconstruction}. Since \(D\) is very ample, after possibly further excluding another finite set of places, all of the hypotheses of Theorem~\ref{thm:Reconstruction} will be satisfied and we deduce the existence of a point \(x\in |D|\) such that \(\spp_v(\kk(x)/k) = \spp'\), as desired.
\end{proof}

%%%%%%%%%%%%%%%%%%%%%%%%%%%%%%%%%%%%%%%%%%%%%%%%%%%%%%%%%%%%%%%%%%%%%%%%%%%%%%%%
%%%%%%%%%%%%%%%%%%				Bibliography			%%%%%%%%%%%%%%%%%%%%%%%%
%%%%%%%%%%%%%%%%%%%%%%%%%%%%%%%%%%%%%%%%%%%%%%%%%%%%%%%%%%%%%%%%%%%%%%%%%%%%%%%%

\end{document}